\documentclass[10pt, reqno]{amsart}

\usepackage{amsmath,amssymb}
\usepackage{hyperref}
\usepackage{graphicx}
\usepackage{xcolor}
\usepackage{multirow}

\usepackage{thmtools}
\usepackage{thm-restate}
\usepackage{mathrsfs}
\usepackage{todonotes}

\newtheorem{lemma}{Lemma}[section]
\newtheorem{proposition}{Proposition}[section]
\newtheorem{theorem}{Theorem}[section]

\newtheorem{remark}{Remark}[section]

\numberwithin{equation}{section}

\newcommand{\cosuff}{\mathcal{Q}}

\newcommand{\mrcv}{\mathop{\hat{\sigma}^2_{\mathrm{mrcv}}}}
\newcommand{\rcv}{\mathop{\hat{\sigma}^2_{\mathrm{rcv}}}}
\newcommand{\pr}{\mathop{\mathrm{pr}}}

\makeatletter
\newcommand*\bigcdot{\mathpalette\bigcdot@{.5}}
\newcommand*\bigcdot@[2]{\mathbin{\vcenter{\hbox{\scalebox{#2}{$\m@th#1\bullet$}}}}}
\makeatother

\newcommand{\EE}{\mathbb{E}}
\newcommand{\RR}{\mathbb{R}}

\usepackage{relsize}
\newcommand{\T}{{\mathsmaller {\rm T}}}

\def\vhalf{\textstyle{\frac 1 2}}

\author{Heather Battey}
\address{Heather Battey, Yanbo Tang: Department of Mathematics, Imperial College London, 180 Queen's Gate, London, SW7 2AZ, UK}	\email{h.battey@imperial.ac.uk, yanbo.tang@imperial.ac.uk}

\author{Daniel Garc{\'i}a Rasines}
\address{Daniel Garc{\'i}a Rasines: Colegio Universitario de Estudios Financieros, Calle Almansa 101, Madrid 28040, Spain.}
\email{daniel.garciarasines@cunef.edu}

\author{Yanbo Tang}

\begin{document}

	\title[]{Post-reduction inference for confidence sets of models}
	
\begin{abstract}
	Sparsity in a regression context makes the model itself an object of interest, pointing to a confidence set of models as the appropriate presentation of evidence. A difficulty in areas such as genomics, where the number of candidate variables is vast, arises from the need for preliminary reduction prior to the assessment of models. The present paper considers a resolution using inferential separations fundamental to the Fisherian approach to conditional inference, namely, the sufficiency/co-sufficiency separation, and the ancillary/co-ancillary separation. The advantage of these separations is that no direction for departure from any hypothesised model is needed, avoiding issues that would otherwise arise from using the same data for reduction and for model assessment. In idealised cases with no nuisance parameters, the separations extract all the information in the data solely for the purpose for which it is useful, without loss or redundancy. The extent to which estimation of nuisance parameters affects the idealised information extraction is illustrated in detail for the normal-theory linear regression model, extending immediately to a log-normal accelerated-life model for time-to-event outcomes. This idealised analysis provides insight into when sample-splitting is likely to perform as well as, or better than, the co-sufficient or ancillary tests, and when it may be unreliable. The considerations involved in extending the detailed implementation to canonical exponential-family and more general regression models are briefly discussed. As part of the analysis for the Gaussian model, we introduce a modified version of the refitted cross-validation estimator of Fan et al. (2012), whose distribution theory is tractable in the appropriate conditional sense.

		\medskip
		
		\emph{Some key words}: Ancillarity; Co-sufficient manifold; Inferential separations; Information geometry; Model uncertainty; Regression; Selective inference; Sparsity.
	\end{abstract}		
	
	\maketitle

	\section{Introduction}\label{secIntro}

For regression problems with many potential explanatory variables, Cox (1968) noted that multiple low-dimensional combinations of variables are often compatible with the data. An arbitrary choice between well-fitting models, while reasonable for prediction, will typically be misleading for scientific understanding, as different models correspond to different scientific explanations. 
This perspective, which was emphasised repeatedly over many years (e.g.~Cox, 1995; Cox and Snell, 1974, 1989), has acquired renewed importance in the light of a multitude of modern areas, notably genomics, in which the number of variables for consideration is vast. Routine practice is to report a single model identified by a variable-selection algorithm, the most widely used being the lasso (Tibshirani, 1996). The recommendation of Cox and Battey (2017), by constrast, was to report all models that are compatible with the data at a chosen significance level, i.e.~a confidence set of models. Some theoretical considerations associated with this proposal under a high-dimensional regime were developed by Battey and Cox (2018) and Lewis and Battey (2025). While there are several papers in the epidemiology literature employing the reduction strategy of Cox and Battey (2017), the more important notion of a confidence set of models is yet to gain traction, although some isolated recent papers have pointed to a need to acknowledge model uncertainty (see \S \ref{secModelUnc}). For some epidemiological examples where sparsity-induced uncertainty has been acknowledged in the form of confidence sets of models, see Kartsonaki \emph{et al.~}(2022) and Prosser \emph{et al.~}(2026).

A set of all sparse models compatible with the data, although superficially less appealing than a single one, conveys more honestly the information in the data. In principle, from the sets of compatible models identified, the conclusions might be narrowed down through a carefully designed experiment, or by way of the contradiction argument laid out in a particular context by Battey and Reid (2023, section 5). Cox and Battey (2017) discussed how one might extract compact messages from large confidence sets of models. Specifically, among the multitude of models reported in their supplementary material, two variables, call them $v_1$ and $v_2$, are present in 96\% and 94\% of cases. In 78\% of the models in which $v_2$ is not present, another variable, $v_3$, is present in its place, and only 1\% of models include neither $v_2$ nor $v_3$.

Suppose that a regression model in $p$ variables is correctly specified but over-conditioned, in the sense that the vector of regression coefficients contains many zero entries, the positions of the non-zeros being unknown. In the remainder of the paper, we therefore equate the specification of a model with a set of variable indices and write  $\mathcal{E}^*\subset [p]$ for the sparse model in $s^*<p$ variables that generated the data. When $p$ is not so large as to make the procedure practically infeasible or theoretically unsound, a simple way of constructing a confidence set of models is via a likelihood-ratio test of every low-dimensional subset of variables $\mathcal{E}_m\subset [p]:=\{1,\ldots,p\}$ against the encompassing model $[p]$. All models $\mathcal{E}_m$, $m=1,\ldots,2^p$ not rejected at level $\alpha$ comprise the confidence set of models $\mathcal{M}$, and this, by standard arguments, contains $\mathcal{E}^*$ with asymptotic probability $1-\alpha$ in hypothetical repeated use. It is sensible in practice to restrict attention only to those models of reasonable size, substantially reducing the number of assessments to be made. 

For the genomics applications we have in mind, $p$ is typically of the order of several thousand, while the number $n$ of individuals is usually in the hundreds or even dozens. The above  construction of $\mathcal{M}$ is therefore infeasible without preliminary reduction from $[p]$ to an encompassing model, $\hat{\mathcal{E}}$ say, of more manageable size. The circumflex emphasises that $\hat{\mathcal{E}}$ is constructed from the data; it is not intended as an estimator of $\mathcal{E}^*$, a much smaller subset. Any reasonable variable screening procedure can be used for the construction of $\hat{\mathcal{E}}$; see Lewis and Battey (2025) for a theoretical discussion of the implications of that choice when the ultimate objective is the construction of $\mathcal{M}$.

The key point as far as the present work is concerned is that, since the comprehensive model $\hat{\mathcal{E}}$ is selected in the light of the data, it fits the observed data better than an arbitrary model of the same size encompassing the one to be tested. For this reason, a likelihood-ratio test of $\mathcal{E}_m \subset \hat{\mathcal{E}}$ against $\hat{\mathcal{E}}$ typically rejects too often in hypothetical repeated use, so that the resulting model confidence set $\mathcal{M}$ has a lower coverage probability than suggested by the nominal level of the test. The problem is most pronounced in the context of logistic regression: in high dimensions there are typically many low-dimensional separating hyperplanes among the candidate variables that perfectly separate cases from non-cases in the observed sample. If just one such set was retained through reduction, a logistic regression fitted to $\hat{\mathcal{E}}$ would have perfect fit, and a likelihood ratio test would be unable to detect any $\mathcal{E}^*$ that did not also form a separating hyperplane in the same set of data.

Battey and Cox (2018) used sample splitting as a simple device for restoring the nominal coverage, although it was noted that a potentially better approach would use a minimal sufficiency separation to assess model adequacy (Barndorff-Nielsen and Cox, 1994), bypassing the likelihood ratio test and obviating sample splitting. Essentially the same inferential separation was used by Engen and Lilleg{\aa}rd (1997), Lindqvist et al.~(2003), Lindqvist and Taraldsen (2005), Lockhart et al.~(2007), and Barber and Janson (2022) to assess model fit, but not in the context of the post-reduction inference considered here, where the case for such separations over likelihood-ratio or other directed tests is much stronger. The approaches taken in the above references are based on Monte Carlo algorithms to approximately sample from the appropriate subset of the sample space after suitable conditioning. Markov Chain Monte Carlo methods are too computationally expensive for present purposes, as construction of $\mathcal{M}$ entails assessing all low-dimensional subsets of $\hat{\mathcal{E}}$ for their compatibility with the data, a combinatorially large problem.

The principles developed in the present paper, specifically for the purpose of confidence sets of sparse regression models, are based on the sufficiency/co-sufficiency separation, and the ancillary/co-ancillary separation fundamental to the Fisherian approach to conditional inference and reviewed in Chapters 2.3 and 2.5 of Barndorff-Nielsen and Cox (1994). In the absence of a proposed direction of departure from the postulated model, a combination of the ancillary and co-sufficient components in principle achieves the maximal extraction of information relevant for the assessment of model adequacy, having extracted the information relevant for estimation and reduction. In practice there is considerable flexibility in how such information is used, the most na{\"i}ve approaches being ineffectual for detecting relevant departures. 

A key question is over the efficacy of sample splitting relative to the more sophisticated approaches developed here. In simpler contexts, methods powerful against specific alternatives are often valuable in small samples. This raises the possibility that using a large portion of the data for reduction and a small portion for model assessment may be as or more effective than the inferential separations considered here, provided that the test for model assessment specifies a direction for departure from the null, and provided that the sample size for model assessment is sufficiently large that any asymptotic distributions hold to an adequate order of approximation. In the context of inference on the a collection of means, selected on the basis of data from normal distributions, Cox (1975) showed via a theoretical analysis that the efficiency of sample splitting relative to an exact conditional approach is surprisingly high. In the present context, there is a clearer separation of information in the data relevant to the different phases of analysis, so a greater efficiency gain might be expected, although this has to be counterbalanced by the lack of directionality for the full-sample test. These considerations are probed by simulation in \S \ref{secSim} for the normal-theory linear regression model. The normal-theory setting is one in which sample splitting has a relatively high efficiency due to the availability of the $F$-test for assessing model adequacy, whose distribution theory is exact and therefore applies at small sample sizes.

\section{Related literature and summary of contributions}\label{secLiterature}

\subsection{Post-selection versus post-reduction inference}\label{secPostSelec}

The only commonality between the post-reduction inference problem considered here and the more familiar post-selection inference problem of e.g.~Lee \emph{et al.}~(2016), Fithian \emph{et al.}~(2017) is that an uncritical re-use of the data leads to erroneous calibration. The most important conceptual difference from the post-selection inference literature is that our inferential goal, the true model, is not data-dependent, as the reduction phase serves only to focus the analysis, and does not dictate the object of inference. Although the analysis is performed under the assumption that $\{\mathcal{E}^* \subset \hat{\mathcal{E}}\}$, conditioning on this event barely affects the distribution of test statistics for model assessment, for the latter satisfies, by construction, $\text{pr}(\mathcal{E}^* \subset \hat{\mathcal{E}})\approx 1$. 

\subsection{Model uncertainty}\label{secModelUnc}

Several other formulations have as their goal the broad objective of reflecting model uncertainty. Most obviously, Bayesian approaches to model selection give rise to Bayesian credible sets of models, whose interpretations are different. In a low-dimensional context, Hansen et al.~(2011) proposed something called a model confidence set, but which does not have the usual properties associated with confidence sets. In particular, Hansen et al.~(2011) consider the relative explanatory power of pairs of models among an initial candidate set, leading to specification of a reduced set of models with equally good or equally poor fit. By contrast, if no sparse model is compatible with the data, a confidence set of models in the sense of Cox and Battey (2017) would be empty. Lei (2020) proposes an approach based on cross-validation that returns a set of well-fitting candidate models and has a coverage probability guarantee for the model that is optimal in a predictive sense. The definition used by Ferrari and Yang (2015) is the same as ours; since they consider a low-dimensional setting, there are no issues of post-reduction inference.  

Breiman (2001) referred to the existence of many alternative predictive models as the \emph{Rashomon effect}, and used this to challenge the relevance of statistical models, a view that we do not share. See Rudin (2025) for additional perspectives. 

A natural proposal for a confidence set of models would involve applying the lasso to bootstrapped samples. This, however, does not have the required coverage properties, as noise variables appear earlier on the lasso path than the last signal variable with high probability (Su \emph{et al.}, 2017; Su, 2018).

\subsection{Summary of contributions}\label{secContributions}

The idea of using the sufficiency/co-sufficiency separation for the avoidance of post-reduction miscalibration of confidence sets of regression models was mentioned briefly elsewhere (Battey, 2022; Battey and Cox 2022). The main contribution of the present paper is to make the relevant inferential separations explicit for regression models, to provide geometric insight into the nature of the inferential problem, and to embed into this framework a generalisation of the randomised inference strategy of Rasines and Young (2023). The work also points to more general insights into the foundations of highly parametrised models, which will be developed elsewhere.

\section{Assessment of model adequacy}\label{secModelSep}

\subsection{Introduction}\label{secPreambleSep}

The discussion to be presented in this section is, at this point, disconnected from the problem outlined in the introduction. The connection will be established in \S \ref{secOperational} after the presentation of some explicit calculations for the Gaussian linear model in \S \ref{secGL}.

\subsection{The co-sufficient manifold}\label{secCosuff}

The role of sufficiency in the assessment of model adequacy, implicit in R. A. Fisher's work, is more easily approached via Barndorff-Nielsen and Cox (1994, p.~29). When the ideas can be operationalised, there are no difficulties associated with a double use of the data for model assessment and parametric inference. For notational simplicity throughout the present \S \ref{secModelSep}, we drop the subscript $m$ indexing the particular model under consideration. The ideas are at this point detached from the specifics of confidence sets of models and follow the geometric exposition of Battey (2024). 

Consider an arbitrary model $\mathcal{E}$, provisionally assumed to contain the true unknown outcome distribution, and having an associated parameter vector $\theta$. The sufficient statistic $S=s(Y)$ for $\theta$ is a function of $Y=(Y_1,\ldots,Y_n)$ with observed value $s^o=s(y^o)\in \RR^d$, where $y^o\in\RR^n$ is the observed value of $Y$. For $n$ independent draws from the same distribution in model $\mathcal{E}$, the joint density or mass function at $y=(y_1,\ldots,y_n)$ separates as
\begin{equation}\label{eqFactorisationThm}
	f_Y(y;\theta)=\prod_{i=1}^n f_{Y_i}(y_i;\theta) = g(s(y);\theta)\prod_{i=1}^n h(y_i)
\end{equation}
by the factorisation theorem (e.g.~Barndorff-Nielsen and Cox, 1994, ch.~2.3). An important implication of \eqref{eqFactorisationThm}, provided that the dimension $d$ of the minimal sufficient statistic $S$ is smaller than $n$, is that there is information in $Y$ not used for inference on $\theta$ and therefore available for assessment of the model.

It is convenient to write the information in $Y$ as $(S,Q(S))$ where, in this notional idealised separation, $Q(S)$ is not a function of $S$ in the conventional sense, and typically does not have an analytic form in terms of the observations; rather $Q(s^o)$ has the conditional distribution of $Y$, or some one-to-one transformation of $Y$, given $S=s^o$. The realisation $s^o$ fixes the co-sufficient manifold 
\begin{equation}\label{eqCoSuffMan}
	\cosuff(s^o)=\{y\in\RR^n: s(y)=s^o\}\subset \RR^n,
\end{equation}
leaving $n-d$ degrees of freedom for variation of $y$ consistent with the constraint $s(y)=s^o$, i.e., $\cosuff(s^o)$ is a manifold of dimension $n-d$ embedded in $\RR^n$. Once $s^o$ is observed, the notional random variable $Q(s^o)$ captures the randomness in $Y$ after conditioning on $S=s^o$ and is constrained to $\cosuff(s^o)$. The implicit general definition of $Q(s^o)$ is
\begin{equation}\label{eqCosuffStat}
	\pr(Q(s^o)\in \mathcal{A})=\pr(Y\in \mathcal{A}|S=s^o),
\end{equation}
where $\mathcal{A}\subseteq \cosuff(s^o)$ by definition.

Barndorff-Nielsen and Cox (1994, p.~29) encourage an understanding of the separation \eqref{eqFactorisationThm} for the assessment of model adequacy by hypothetically treating the data as arising in two stages: in the first stage we learn that $S=s^o$ which is used for inference on $\theta$. In the second stage we learn that $Y=y^o$, which is effectively one draw from the conditional distribution of $Y$ given $S=s^o$. If $y^o$ is extreme when calibrated against this conditional distribution, that casts doubt on the adequacy of the model. This is the usual analogue of proof by contradiction underpinning the core Fisherian hypothesis tests. By calibrating $y^o$ against the conditional distribution under the model hypothesis of $Y$ given $S=s^o$, we avoid the need to compare a smaller model against a larger one, thereby avoiding the post-reduction miscalibration suffered by a likelihood ratio test against a data-dependent encompassing model, as described in \S \ref{secIntro}.

For some models, notably the linear regression model and canonical generalised linear models, the co-sufficient manifold can be specified explicitly. When this is not the case, it is natural to seek local approximations to the manifold of interest, and since $s(y^o)=s^o$ by construction, it is reasonable to consider the tangent plane to $\cosuff(s^o)$ locally at $y^o$. Some discussion of this is provided in \S \ref{secTangent}. We focus in the present paper on situations for which the co-sufficient manifold is readily deduced.

\subsection{The ancillary separation}\label{secAncill}

A second approach to the assessment of model adequacy is available when the dimension $d$ of $S$ is larger than $d_\theta$, the dimension of $\theta$. Write $S\cong (S_1,A)$, where $S_1$ estimates $\theta$ and $A$ is an ancillary statistic. Ideally $A$ is ancillary in the strongest sense, having a distribution that does not depend on $\theta$, which implies the factorisation
\begin{equation}\label{eqAncillFact}
	f_{S}(s;\theta)= f_{S\mid A}(s \mid a;\theta)f_A(a)
\end{equation}
More generally, definitions of ancillarity in the presence of nuisance parameters are collated on page 38 of Barndorff-Nielsen and Cox (1994). Since the information relevant for inference for $\theta$ is contained in the conditional distribution of $S$ given $A=a^o$, the marginal distribution of $A$ is available for model checking. 

For the maximal amount of information to be extracted for model assessment, the random variables $A$ and $Q(s^o)$ should in principle both be sought. Whether this leads to more powerful inference in practice than an approach based on sample splitting, or based on one or other of $A$ and $Q(s^o)$ individually, is unclear from this abstract discussion because of the possibility of using directed tests when the sample is partitioned, and because the detailed implementation may vary according to the structure of the regression model under consideration. Section \ref{secGL} gives analytic calculations and methodology for the Gaussian linear regression model. 

\section{Distributional calculations for Gaussian linear regression}\label{secGL}

\subsection{Distribution on the co-sufficient manifold}\label{secExplicitGaussian}

Consider a linear regression model in which, under the postulated model, $Y\in \mathbb{R}^n$ is normal with mean $\mu = X\theta \in \mathcal{X}$ and covariance matrix  $\Sigma = \sigma^2 I_n$. The space $\mathcal{X}$ is the $d_\theta$-dimensional linear space spanned by the $d_\theta$ columns of $X$, with $(n-d_\theta)$-dimensional orthogonal complement
\[
\mathcal{X}^\perp = \{v\in \mathbb{R}^n: v^\T x=0,\, x\in \mathcal{X}\}\subset \mathbb{R}^n,
\] 
so that $\mathcal{X}\oplus \mathcal{X}^\perp = \RR^n$.

A minimal sufficient statistic when $\sigma^2$ is unknown is either $S=(X^\T Y,\hat{\varepsilon}^\T\hat{\varepsilon})$ if the covariate matrix $X$ is treated as fixed, or $S=(\hat{\theta},\hat{\varepsilon}^\T\hat{\varepsilon}, X^{\T} X)$ if $X$ is treated as random, where $\hat\theta$ is the ordinary least squares estimator of $\theta$ (also the maximum likelihood estimator) and $\hat{\varepsilon}^\T\hat{\varepsilon}$ is the residual sum of squares. Here we treat $X$ as fixed so that $d=d_\theta+1$. The information in the data that is not contained in the sufficient statistic can be extracted through orthogonal projection of $Y$ onto $\mathcal{X}^\perp$, followed by an additional conditioning on the realised value of $\hat{\varepsilon}^\T\hat{\varepsilon}$. The first conditioning implicit in the projection restricts to the $(n-d_\theta)$-dimensional linear subspace $\mathcal{X}^\perp$, while the second conditioning restricts the lengths of permissible vectors, so that the co-sufficient manifold is a hypersphere embedded in $\mathcal{X}^\perp$, which can also be viewed as an $(n-d)$-dimensional sphere embedded in $\RR^n$. Depending on context, we will refer to the co-sufficient manifold interchangeably as a hypersphere in $\mathcal{X}^\perp$ and an $(n-d)$-dimensional sphere in $\RR^n$, the term \emph{hypersphere} meaning that there is one dimension missing from the ambient space. Having restricted to $\mathcal{X}^\perp$, we can for most purposes think of $\cosuff(s^o)$ as being a unit hypersphere in $\RR^{n-d_\theta}$, except that the coordinates within the larger space $\RR^n$ are specified by the orientation of $\mathcal{X}^\perp$. The distribution on the sphere is uniform under correct specification of the model, as shown by the following calculation, analogous to ones used in the invariant testing literature (e.g.~Kariya, 1980; King, 1980). 

The subspaces $\mathcal{X}$ and $\mathcal{X}^\perp$ are groups under addition, which implies that for any matrix $U$ with column span $\text{Im}(U)= \mathcal{X}^\perp$, the statistic
\begin{equation}\label{eqInvariant}
	q(y,U)= U^\T y /\|U^\T y\| = U^\T y/(y^\T U U^\T y)^{1/2}.
\end{equation}
is invariant under the group with action $g(a,b):y\mapsto ay + b$, where $a>0$, $b\in \mathcal{X}$. The lowest-dimensional specification of $q(y,U)$ containing the same information is obtained by taking the columns of $U$ to be a linear basis for $\mathcal{X}^\perp$. The natural choice of $U$ is the $n\times (n-d_\theta)$ matrix of orthonormal eigenvectors corresponding to the unit eigenvalues of the standard orthogonal projection $M=I_{n} - X(X^\T X)^{-1} X^\T$, the remaining eigenvalues being zero. It follows that $UU^\T = M$ and $U^\T U = I_{n-d_\theta}$. 

Consider $W=w(Y)=U^\T Y$, which is normally distributed with mean zero and covariance $U^\T \Sigma U=\sigma^2 I_{n-d_\theta}$ under the postulated model. In the transformation $w\mapsto (w/\|w\|,w^\T w) =: (q, r^2)$, the volume element transforms as
\begin{equation}\label{eqVol}
	dw_1\cdots dw_{n-d_\theta} = \frac{(r^2)^{\tfrac{n-d_\theta}{2}-1}dr^2 \textstyle{\prod_{j=1}^{n-d_\theta-1}}dq_{j}}{2 \bigl(1-\textstyle{\sum_{j-1}^{n-d_\theta-1}}q_{j}^2\bigr)^{1/2}} =  \vhalf (r^2)^{\tfrac{n-d_\theta}{2}-1}dr^2 d\text{Vol}^{n-d}(q),
\end{equation}
where 
\[
d\text{Vol}^{n-d}(q) = d\text{Vol}^{n-d_\theta - 1}(q) =\frac{\textstyle{\prod_{j=1}^{n-d_\theta-1}}dq_{j}}{\bigl(1-\textstyle{\sum_{j-1}^{n-d_\theta-1}}q_{j}^2\bigr)^{1/2}}
\]
is the natural volume element on the surface of the hypersphere in $\RR^{n-d_\theta}$. The joint density function is therefore
\[
\frac{(r^2)^{(n-d_\theta)/2 - 1}\exp\Bigl(-\frac{r^2}{2\sigma^2} q^\T q\Bigr)}{2 (2\pi\sigma^2)^{(n-d_\theta)/2}}dr^2 d\text{Vol}^{n-d}(q),
\]
from which the marginal density function of $Q$ is obtained by integration using a change of variables from $r^2$ to $r^2 q^\T q/2\sigma^2$ as 
\begin{equation}\label{eqQ}
	\frac{\Gamma\bigl(\frac{n-d_\theta}{2}\bigr)(q^\T  q)^{-(n-d_\theta)/2}}{2\pi^{(n-d_\theta)/2}}d\text{Vol}^{n-d}(q),
\end{equation}
showing that $Q$ is uniformly distributed on the unit hypersphere in $\mathcal{X}^\perp$. Since, under the assumed mean model $\mu \in \mathcal{X}$, $R^2=W^\T W$ is distributed as $\sigma^2$ times a standard $\chi^2$ random variable with $n-d_\theta$ degrees of freedom, independently of $Q$, the conditional density function of $Y$ given $R^2=r^2$ is equal to the marginal density function of $Q$, which is explicitly defined in this case.

This calculation can be generalised to allow for variance components in the form $\Sigma = \sigma^2(I_n + V(\alpha))$ where $V(\alpha)$ is a known matrix function of a vector parameter $\alpha=(\alpha_1, \ldots, \alpha_s)^\T$ with $V(0) = 0$. In this case, the appropriate analogue of \eqref{eqQ} is given by equation (10) of Battey and McCullagh (2024). Alternatively, a transformation to uniformity on the unit hypersphere in $\mathcal{X}^\perp$ is $q\mapsto A_\alpha^{-1/2}q/(q^\T A_{\alpha}^{-1} q)^{1/2}=:Q'$, where $\sigma^2 A_\alpha=U^\T \Sigma U$. %There are many tests of uniformity on the sphere, whose properties are well established in the literature, by contrast, tests for compatibility with equation (10) of Battey and McCullagh (2024), if available, are not mainstream.

There are some fundamental and operational difficulties in applying the model assessment strategies of \S \ref{secModelSep} directly. These are discussed in \S \ref{secOperational} where we provide a more elaborate blueprint for implementation of this broad Fisherian program.

\subsection{Distribution of the ancillary component}\label{secElliptic}

Notice that $R^2$ is the residual sum of squares $\hat\varepsilon^\T \hat\varepsilon$, which is part of the minimal sufficient statistic $S$ and has a distribution independent of $\theta$. It is ancillary for $\theta$ according to the more general definitions on page 38 of Barndorff-Nielsen and Cox (1994) and is therefore also available for the assessment of model adequacy via a version of \eqref{eqAncillFact} allowing for nuisance parameters.

Suppose initially that $\sigma^2$ and the mean subspace $\mathcal{X}$ are known, then $R^2$ carries all the information relevant for distinguishing, say, normal errors from Student-$t$ errors. This follows from the properties of elliptically symmetric distributions, for which the generic form of the density function is 
\[
f_Y(y;\mu, \Sigma, g) \propto |\Sigma|^{-1/2}g\bigl((y-\mu)^\T \Sigma^{-1}(y-\mu) \bigr),
\]
the generator $g:\RR^+ \rightarrow \RR^+$ identifying the member of the elliptically symmetric class. To provide insight, and a basis for some arguments later in the paper, we briefly consider this setting.

Let $U$ be as defined in \S \ref{secExplicitGaussian}. Then $W=U^\T Y$ is a member of the elliptically symmetric class of distributions in $n-d_\theta$ dimensions, with zero mean and covariance $\Sigma_W=\sigma^2 I_{n-d_\theta} = U^\T \Sigma U$ (e.g.~Muirhead, 1982), and therefore has a stochastic representation (Cambanis \emph{et al.}, 1981, Theorem 1) $W\stackrel{d}{=}Z\Sigma_W^{1/2} Q'$, where $\stackrel{d}{=}$ denotes equality in distribution and $Q'$ is uniformly distributed on the surface of the unit hypersphere in $\RR^{n-d_\theta}$. The distribution of $Z$ with density function $f_Z$ determines the member of the elliptic class and is related to the generator $g$ for the standardised variable $\Sigma^{-1/2}_W W$ as (Fang \emph{et al.}, 1990)
\[
f_z(z) =\frac{2 \pi^{(n-d_\theta)/2}}{\Gamma((n-d_\theta)/2)}z^{n-d_\theta-1}g(z^2). 
\]
In the Gaussian case of \S \ref{secExplicitGaussian}, the random variable $R$ satisfies $R\stackrel{d}{=}\sigma Z$ where $Z$ is distributed as the square root of a $\chi^2$ random variable with $n-d_\theta$ degrees of freedom under the null model $\mu\in \mathcal{X}$. As an example, for the Student-$t$ family of errors with $\nu$ degrees of freedom, the distribution of $Z$ is (Kano, 1994) $Z\stackrel{d}{=}\surd(\chi^2_{n-d_\theta}/V)$, where $V\stackrel{d}{=}\chi^2_{\nu}/\nu$ is independent of the numerator.

The random variable $Z^2$ thus carries information on the distribution of the errors when the mean model is known. Conversely, $Z^2=R^2/\sigma^2$ contains valuable information for assessment of the mean model under the assumption that the errors are Gaussian. 

\section{An operational blueprint for Gaussian linear regression}\label{secOperational}

\subsection{Co-sufficiency and confidence sets of models: difficulties}

Sections \ref{secModelSep} and \ref{secGL} covered, respectively, general principles and the corresponding calculations of the normal-theory linear model. These sections were detached from the problem of post-reducton inference for confidence sets of models introduced in \S \ref{secIntro}, to which we now return. 

From an encompassing model $\hat{\mathcal{E}}$ computed using all available data, we consider all submodels $\mathcal{E}_m$ of reasonable size, seeking to assess each one for compatibility with the data. Thus for each submodel $\mathcal{E}_m$, we can construct one realisation $q^o$ of the random variable $Q$ from \S \ref{secExplicitGaussian}. The dependence of $Q$ on the model $\mathcal{E}_m$ is not made explicit in the notation; nor is that of the corresponding null space $\mathcal{X}^\perp$ or the unit hypersphere $\mathcal{Q}(s^o)$ on which the distribution of $Q$ is supported. Under the correct model $\mathcal{E}^*$, the outcome $Y$ projected onto the corresponding $\mathcal{Q}(s^o)$ is uniformly distributed on its surface.  

The difficulty is that a test based on $q^o$ alone has no power in an assessment of model adequacy, as any single unit vectors on $\mathcal{Q}(s^o)$ is compatible with uniformity of distribution. Although power in certain directions could be restored by consideration of alternative models for comparison, this leads to a test statistic essentially equivalent to the $F$-test, i.e.~the likelihood ratio test for Gaussian models, and therefore back to the initial difficulties discussed in \S \ref{secIntro}. Note, however, that even with two replicate observations of the random variable $Q$, there would be a degree of power against uniformity of distribution on $\cosuff(s^o)$, this tending to increase with the number of replicate observations. 

This observation points to an initial suggestion for inference based on co-sufficient information partitioning. Information partitioning, while bearing some superficial similarity to sample splitting, differs in that no observational units are discarded in any of the inferential steps. Since all the data are used for both aspects of inference, reduction and model assessment, no information is relinquished in the reduction to $\hat{\mathcal{E}}$, and virtually no co-sufficient information is lost in the assessment of model adequacy. As it turns out, the simplest partitioning approach, while conceptually revealing, is ineffectual for reasons given in Appendix \ref{secPartitioning}. Instead, we propose in \S \ref{secRandomisation} a more elaborate reformulation, adapting and extending the approach of Rasines and Young (2023), introduced in that work for a different purpose.

\subsection{Co-sufficient information partitioning through randomisation} \label{secRandomisation}

In the vein of Rasines and Young (2023), we first generate independent pseudo-replicates of $Y$ with common distribution $N(\mu, k\sigma^2 I_n)$. The idealised but infeasible randomisation procedure is as follows. Let $L$ be an $n\times (k-1)$ matrix, independent of $Y$, with independent $N(0, 1)$ entries, and define the $n\times k$-dimensional matrix of pseudo-replicates as
\begin{equation} \label{EQ: randomisation}
	[\tilde{Y}^{(1)} \cdots \tilde{Y}^{(k)}] = [Y \; L] \Gamma, 
\end{equation}
where $\Gamma$ is a $k\times k$ matrix of coefficients tuned to obtain the desired distribution as
\begin{equation}
	\Gamma = \begin{bmatrix}
		1 & 1 & 1 & \cdots & 1 \\
		a_1 & -b_1 & -b_1 &  \cdots & -b_1 \\
		0 & a_2 & -b_2 &  \cdots & -b_2 \\
		\vdots & \vdots & \vdots & \vdots & \vdots \\
		0 & 0 & 0 & a_{k-1} & -b_{k-1}
	\end{bmatrix}.
\end{equation}
Here $\{a_i, b_i\colon i = 1, \ldots, k-1\}$ are deterministic coefficients that can be computed recursively as: $a_i = \sigma\tilde a_i$ and $b_i = \sigma\tilde b_i$, where
\begin{align}
	\begin{split}\label{eqRecursive}
		\tilde a_1 =& \; \sqrt{k-1}; \hspace{3.2cm} \tilde b_1 = \frac{1}{\tilde a_1}; \\
		\tilde a_{i+1} =& \; \Bigl(k - 1 -  \sum_{j = 1}^{i} \tilde b_j^2\Bigr)^{1/2}; \hspace{0.98cm}
		\tilde b_{i+1} = \frac{1 + \sum_{j = 1}^{i} \tilde b_j^2}{\tilde a_{i+1}}.
	\end{split}
\end{align}
\begin{proposition}
	\label{lemmaRand}
	If $Y\sim N(\mu, \sigma^2 I_n)$, then $\tilde{Y}^{(1)}, \ldots, \tilde{Y}^{(k)}\sim N(\mu, k \sigma^2 I_n)$ independently.    
\end{proposition}
A proof is provided in Appendix \ref{appLemmaRand}. The above construction can be generalised to allow for any known positive-definite covariance matrix $\Sigma$ in the distribution of $Y$ by adjusting $\Gamma$ appropriately. 

The randomisation scheme is a direct generalisation of the construction introduced by Rasines and Young (2023) in a different context. In particular, Rasines and Young (2023) decompose a random vector $Y\sim N(\mu, \sigma^2 I_n)$ into two vectors of the same size
\begin{eqnarray*}
	Y^{(1)} =& Y + \gamma \sigma L  &\sim  N(\mu, (1+\gamma^2)\sigma^2I_n), \\
	Y^{(2)} =& \;\,Y - \sigma L/\gamma &\sim  N(\mu, (1+\gamma^{-2})\sigma^2I_n), 
\end{eqnarray*}
where $L\sim N(0, 1)$ and $\gamma > 0$ calibrates the variance allocation. On applying this argument recursively $k-1$ times, we have, with $\gamma_1 = (k-1)^{1/2}$, a decomposition of $Y$ into $Y^{(1)} \sim N(\mu, k \sigma^2I_n)$ and $V_1 \sim N(\mu, (k/k-1) \sigma^2I_n)$, and with $\gamma_j = (k-j)^{1/2}$ a decomposition of $V_{j-1}$ into $Y^{(j)}\sim N(\mu, k \sigma^2 I_n)$ and $V_j \sim N(\mu, (k/k-j) \sigma^2I_n)$. The matrix $\Gamma$ emerges on writing $Y^{(i)}$ in terms of $Y$ and the Gaussian errors, and collecting coefficients. The same construction can be obtained by an application of Proposition 2 in Dharamshi et al.~(2025) but the explicit form of $\Gamma$ that achieves a decomposition into equal-variance components is new.

From the pseudo-replicates of the response vector, we obtain independent replicates of $Q$, constructed for any given model as $\tilde Q^{(i)}=U^\T\tilde Y^{(i)}/\|U^\T\tilde Y^{(i)}\|$ for $i = 1, \ldots,k$, where $U$ is that of \S \ref{secExplicitGaussian}. Under correct submodel specification, $\mu = X\theta$, these $k$ replicates are independent and uniformly distributed on the unit hypersphere in $\mathcal{X}^\perp$.

If $\sigma^2$ is known or precisely and accurately estimated, then a larger choice of $k$ is intuitively better, as this plays the role of sample size for model assessment, determining the number of observations on the sphere. Since an actionable version of the randomisation approach requires estimation of $\sigma^2$, which induces some asymptotically negligible dependence between pseudo replicates, small $k$ is better at small sample sizes, a point to which we return below.

The classical estimator, constructed from the residual sum of squares of the linear model including all covariates in the encompassing set $\hat{\mathcal{E}}$, is consistent provided that $n - \vert \hat{\mathcal{E}} \vert \to \infty$. Since $\hat{\mathcal{E}}$ tends to overfit to the data  in small samples, it typically underestimates $\sigma^2$ unless the sample size is large. Small-sample bias is reduced through the use of the refitted cross-fitting procedure of Fan \emph{et al.}~(2012), who provided statistical guarantees. Below we propose a minor adaptation of this estimator, whose distribution theory is simple modulo the necessary conditioning arguments, and convenient for the Gaussian linear model. 

The estimator described by Fan et al (2012) partitions the outcome and covariate data into two disjoint subsets, $(Y^{(1)}, X^{(1)})$ and $(Y^{(2)}, X^{(2)})$, each of size $n/2$. Fan et al.~(2012) proposed to fit a conservative variable selection procedure within each subset, such that the true set $\mathcal{E}^*$ is contained in both selected sets $\hat{\mathcal{E}}^{(1)}$ and $\hat{\mathcal{E}}^{(2)}$ with high probability. Let $\mathcal{X}^{(1)}_2$ be the space spanned by the columns of $X^{(1)}$ indexed by $\hat{\mathcal{E}}^{(2)}$ and let $\mathcal{X}^{(2)}_1$ be the space spanned by the columns of $X^{(2)}$ indexed by $\hat{\mathcal{E}}^{(1)}$. The projection matrices onto these two subspaces are written $P^{(1)}_2$ and $P^{(2)}_1$. The refitted cross-validation estimator $\hat\sigma^2_{\text{rcv}}$ of $\sigma^2$ is then given by the average of the two independent estimators 
\[
\hat{\sigma}_1^2=\frac{\|(I-P^{(1)}_2)Y^{(1)}\|_2^2}{n/2-|\hat{\mathcal{E}}^{(2)}|}, \quad \quad\hat{\sigma}_2^2=\frac{{\|(I- P^{(2)}_1)Y^{(2)}\|_2^2}}{n/2-|\hat{\mathcal{E}}^{(1)}|}.
\]
%This assumes that the dimension of the subspace $\mathcal{X}^{(j)}_{-j}$ is equal to $|\hat{\mathcal{E}}^{(-j)}|$, i.e.~that the matrix formed from the columns of $X^{(j)}$ indexed by $\hat{\mathcal{E}}^{(-j)}$ is full rank. 
Empirically, the best approach computes this estimate, or the one below, using only $[\gamma n]$ observations for some $0.5 < \gamma < 1$, so that $n$ in the previous display is replaced by $[\gamma n]$. This further limits the dependence between $\rcv$ and $U^TY$, resulting in a better calibration of the test statistic under the null distribution. In simulations, we found that $\gamma=0.6$ was a good choice across a broad range of scenarios. 

In order to prove the distribution theory of the test statistic in Theorem \ref{corrRayleigh}, it is convenient to use a slightly modified version of the refitted cross validation estimator. Let $df_1 = n/2-|\hat{\mathcal{E}}^{(2)}|$ and $df_2 = n/2-|\hat{\mathcal{E}}^{(1)}|$ and define the modified refitted cross-validation estimator as
\begin{equation}\label{eqMRCV}
	\hat{\sigma}^2_{\text{mrcv}}:= \frac{df_1 \hat{\sigma}^2_1 + df_2 \hat{\sigma}^2_2}{df_1 + df_2}. %\sim \sigma^2 \frac{\chi^2_{df_1 + df_2}}{df_1 + df_2}
\end{equation}

Proposition \ref{propMRCV} establishes a distributional approximation to the distribution of this estimator when the sizes of the compassing models, $|\hat{\mathcal{E}}^{(1)}|$ and $|\hat{\mathcal{E}}^{(2)}|$, are fixed independently of the data. Although in the case of fixed $|\hat{\mathcal{E}}^{(1)}|$ and $|\hat{\mathcal{E}}^{(2)}|$ they would typically be chosen equal so that $\mrcv$ reduces to $\rcv$, simulations (not reported) show that the distributional approximation given in Proposition \ref{propMRCV} is almost exact at moderate sample sizes even when $|\hat{\mathcal{E}}^{(1)}|$ and $|\hat{\mathcal{E}}^{(2)}|$ are data-dependent.

\begin{proposition}\label{propMRCV}
	Let $\nu = df_1 + df_2$. Let $|\hat{\mathcal{E}}^{(1)}|$ and $|\hat{\mathcal{E}}^{(2)}|$ be fixed independently of the data and satisfy $\vert\hat{\mathcal{E}}^{(1)} \vert + \vert\hat{\mathcal{E}}^{(2)} \vert= o(n)$. Assume that for any set $E$ such that $E\subseteq \{1, \ldots, p\}$ and $\vert E\vert =  \vert\hat{\mathcal{E}}^{(-j)} \vert$, $X^{(j)}_{E}$ has full column rank for $j = 1, 2$, where the notation $-j$ denotes the complementary index. Assume further that $
	\pr(\mathcal{E}^*\not\subseteq \hat{\mathcal{E}}^{(1)} \cap \hat{\mathcal{E}}^{(2)}) = O(n^{-1/2})$. Then,
	\[
	\hat{\sigma}^2
	\stackrel{d}{=}
	\sigma^2\,\frac{\chi^2_{\nu}}{\nu}
	+
	O_p( n^{-1/2}), \quad \text{as } n\to\infty,
	\]
	where $\chi^2_{\nu}$ is a chi-square random variable with $\nu$ degrees of freedom.
	
\end{proposition}

Let $\hat{Y}^{(1)}, \ldots, \hat{Y}^{(k)}$ be the $k$ columns of the $n\times k$ matrix obtained by replacing $\Gamma$ in \eqref{EQ: randomisation} by $\hat\Gamma$, the realisable version of $\Gamma$ with $\hat\sigma$ in place of $\sigma$. A recommendation is $\hat\sigma_{\text{mrcv}}$ computed on a portion $[0.6n]$ of the data. With $\hat Q^{(i)}=U^\T\hat Y^{(i)}/\|U^\T\hat Y^{(i)}\|$, Proposition \ref{thmQ} quantifies the maximum Euclidean distance between $\hat Q^{(1)},\ldots, \hat Q^{(k)}$ and the idealised versions $ \tilde Q^{(1)},\ldots, \tilde Q^{(k)}$ that treat $\sigma$ as known. The conclusion is in terms of $k$ and the error rate $a_n$ of the estimator $\hat\sigma$ of $\sigma$. It should be kept in mind that, although the notation does not reflect this, the construction of both $\hat Q^{(1)},\ldots, \hat Q^{(k)}$ and $\tilde Q^{(1)},\ldots, \tilde Q^{(k)}$ depends on a postulated model $\mathcal{E}_m$, inducing a mean subspace, $\mathcal{X}_m$ say. 

\begin{proposition}\label{thmQ}
	Let $\hat\sigma$ be an estimator satisfying $|\hat\sigma - \sigma| = O_p(a_n)$ and take $k\lesssim a_n^{-2}$. Under the null hypothesis $\mathcal{X}=\mathcal{X}_m$,
	\[ 
	\max_{i \in [k]}\|\hat Q^{(i)} - \tilde Q^{(i)}\|_{2}= O_p(a_n k^{1/2}). 
	\]
\end{proposition}

For the refitted and modified refitted cross-validation estimators of error variance, $a_n=n^{-1/2}$, which suggests an upper growth rate of $k$ with $n$ of order $k\lesssim n^{1/2}$.

Proposition \ref{thmQ}, although reassuring, misses an important point. The relevant metric in which to measure differences between $\hat Q^{(1)},\ldots, \hat Q^{(k)}$ and $\tilde Q^{(1)},\ldots, \tilde Q^{(k)}$ depends on the way in which they are used to assess model adequacy, which is via their distribution under the null hypothesis. Since the weak topology is weaker than the norm topology, we can expect better results than those reflected by Proposition \ref{thmQ}. In particular, an intuitive reaction to the discussion of \S \ref{secElliptic} is that the distribution of $\tilde Q^{(1)},\ldots, \tilde Q^{(k)}$ should be be relatively unaffected by estimation of $\sigma$. The discussion to follow, and the proof of Theorem \ref{corrRayleigh} verifies this intuition, and clarifies some more subtle points.

In contrast to the simple data-partitioning scheme considered in \S~\ref{secPartitioning}, the replicates of $Q$ obtained through randomisation via \eqref{EQ: randomisation} are approximately independent and identically distributed  under mean models other than the postulated one $\mu = X\theta$. As a consequence, this decomposition is more amenable than that of \S~\ref{secPartitioning} to standard testing procedures. While there are many tests of uniformity on a sphere, we focus on the high-dimensional Rayleigh test, as its properties are guaranteed under a high-dimensional regime that matches our requirements, and efficient implementations are available in standard software. 

In the present context, the Raleigh test is based on the statistic
\begin{equation}\label{eqRn}
	R_n := \frac{\sqrt{2(n - d_{\theta})}}{k} \sum_{1 \leq i < j \leq k} \left< \frac{U^\T \tilde Y^{(i)}}{\lVert U^\T \tilde Y^{(i)} \rVert_2} , \frac{U^\T \tilde Y^{(j)}}{\lVert U^T \tilde Y^{(j)} \rVert_2} \right> \rightarrow_{d} N(0, 1),
\end{equation}
or rather, its analogue with $\hat{Y}^{(i)}$ in place of $\tilde{Y}^{(i)}$. This test rejects for large values of $R_n$, calibrated against the asymptotic null distribution. This is consistent with a direction-free alternative hypothesis: if $\mu\notin \mathcal{X}$, then $\tilde Q^{(1)},\ldots, \tilde Q^{(k)}$ are approximately a random sample from projected Gaussian with non-zero mean, so small angles among them are indicative of model misspecification. The distributional convergence was ascertained by Paindaveine and Verdebout (2016) in a more general setting than that needed here. Since, under the null hypothesis that the postulated mean model is correct, $U^\T \tilde{Y}^{(i)}$ is normally distributed with mean zero and variance $k\sigma^2 I_{n-d_\theta}$, exact distribution theory for $R_n$ can in principle be derived from the calculations of Fisher (1915). We have
\begin{equation}\label{eqStandardised}
	\frac{U^\T \tilde Y^{(i)}}{\lVert U^\T \tilde Y^{(i)} \rVert_2} = \frac{\sigma^{-1} k^{-1/2} U^\T \tilde Y^{(i)}}{\lVert \sigma^{-1}k^{-1/2} U^\T \tilde Y^{(i)} \rVert_2},
\end{equation}
and the standardised quantity in the numerator is standard normally distributed in $n-d_\theta$ dimensions. It follows that the distribution of 
\[
\left< \frac{U^\T \tilde Y^{(i)}}{\lVert U^\T \tilde Y^{(i)} \rVert_2} , \frac{U^\T \tilde Y^{(j)}}{\lVert U^T \tilde Y^{(j)} \rVert_2} \right> = \cos\left(\angle \left(\frac{U^\T \tilde Y^{(i)}}{k^{1/2}\sigma}, \frac{U^\T \tilde Y^{(j)}}{k^{1/2}\sigma}\right) \right) =: \bar{R},
\]
is that of the sample correlation coefficient between two standard normally distributed random variables from $n-d_\theta$ observation points. In the above display, $\angle(a,b)$ denotes the angle between vectors $a$ and $b$. The density function of the sample correlation coefficient $\bar{R}$ is (Fisher, 1915)
\begin{equation}\label{eqFisher1915}
	f_{\bar{R}}(r) =\frac{\Gamma((n-d_\theta)/2)}{\sqrt{\pi}\Gamma\{(n-d_\theta-1)/2\}} (1-r^{2})^{(n-d_\theta-3)/2}, \quad -1<r<1.
\end{equation}
That the distribution of $R_n$ in \eqref{eqRn} is asymptotically standard normal follows essentially by the central limit theorem on noting that $\mathbb{E}(\bar{R})=0$ and
\begin{align*}
	\text{var}(\bar{R})  = & \;\frac{\Gamma(\tfrac{n-d_\theta}{2})}{\sqrt{\pi}\Gamma(\tfrac{n-d_\theta-1}{2})} \int_{0}^{1} s^{1/2}(1-s)^{(n-d_\theta-3)/2}ds \\
	%&=& \frac{\Gamma(\tfrac{n-d_\theta}{2})B(\tfrac{3}{2},\tfrac{n-d_\theta-1}{2})}{\sqrt{\pi}\Gamma(\tfrac{n-d_\theta-1}{2})} 
	=& \;  \frac{\Gamma(\tfrac{n-d_\theta}{2})}{2\Gamma(\tfrac{n-d_\theta}{2}+1)} \simeq \frac{1}{n-d_\theta}, \quad (n-d_\theta\rightarrow \infty),
\end{align*}
where the equality in the second line uses the gamma recursion $\Gamma(3/2)=\sqrt{\pi}/2$, and the asymptotic form comes from Stirling's approximation $\Gamma(x+a)/\Gamma(x)\simeq x^a$ for $a$ fixed and $x$ large. In fact, for the Gaussian case, the asymptotic distribution is essentially exact for $k$ as small as 2, because in \eqref{eqFisher1915}
\begin{equation}\label{eqFisherApprox}
	(1-r^2)^{(n-d_\theta - 3)/2}\simeq e^{-(n-d_\theta)r^2/2}, \quad (r\rightarrow 0),
\end{equation}
which is the non-normalised density function of a Gaussian random variable with mean zero and variance $(n-d_\theta)^{-1}$.

Theorem \ref{corrRayleigh} shows how estimation of $\Gamma$ using the modified refitted cross-validation estimator \eqref{eqMRCV} affects the theoretical guarantee \eqref{eqRn}. 

\begin{theorem}\label{corrRayleigh}
	Under the conditions of Proposition \ref{propMRCV}, 
	\begin{align*}
		\hat{R}_n := & \; \frac{\sqrt{2(n - d_\theta)}}{k} \sum_{1 \leq i < j \leq k} \biggl< \frac{U^\T \hat{Y}^{(i)}}{\lVert U^\T \hat{Y}^{(i)} \rVert_2} , \frac{U^\T \hat{Y}^{(j)}}{\lVert U^\T \hat{Y}^{(j)} \rVert_2} \biggr> \\
		\stackrel{d}{=} & \; R_n + 
		O_{p}\bigl(\max\{k\nu^{-3/2}, n^{-1/2}\}\bigr), \quad (n, \nu \rightarrow \infty,),
	\end{align*}
	where $\nu=df_1+df_2 \asymp n$. It follows from \eqref{eqRn} that $\hat{R}_n \rightarrow_d N(0,1)$ provided that $k\lesssim n^{1/2}$.
\end{theorem}

\begin{remark} 
	Although the estimator of error variance looks to play a minimal role in Theorem \ref{corrRayleigh}, it is important for the proof strategy that the modified cross validation estimator $\mrcv$ is used. 
\end{remark}

The error term is minimised with $k=2$ and, in view of \eqref{eqFisherApprox}, we expect this choice to provide a close distributional approximation under the null model. There may be reasons to prefer larger $k$, particularly if this provides appreciable improvements in power to detect departures from the null model for only a small loss in coverage. Conditioning on $|\hat{\mathcal{E}}^{(1)}|$ and $|\hat{\mathcal{E}}^{(2)}|$ in the statement of Theorem \ref{corrRayleigh} breaks the dependence between $\hat Q^{(1)},\ldots, \hat Q^{(k)}$. In the simulations of \S \ref{secSim}, where this aspect is implicitly averaged over in the simulation replicates, some mild dependence is introduced, which may contribute to the deterioration in the distribution theory for increasing $k$.

It may be questioned, in view of the discussion preceding Theorem \ref{corrRayleigh}, why the randomisation of equation \eqref{EQ: randomisation} is needed at all, as rotational invariance of the Haar measure implies that any vector $v$ on a hypersphere in $\RR^{n-d_\theta}$ gives rise to an angle $\langle Q, v \rangle$ that follows the distribution in \eqref{eqFisher1915}. Such an approach, like that of \S~\ref{secPartitioning} ignores the orientation of the embedding manifold $\mathcal{X}^\perp$ in $\RR^n$. Thus, while the null distribution theory is correct in an unconditional sense, the analysis is not relevant for detecting departures from the null model.

\subsection{A test based on the ancillary component}\label{secAncillaryTest}

As discussed in \S \ref{secExplicitGaussian} and \S \ref{secElliptic}, $R^2 = \hat\varepsilon^\T \hat\varepsilon$ has a $\sigma^2 \chi^2_{n-d_\theta}$ distribution and is ancillary for $\theta$ under the postulated mean model. To exploit this for model assessment, a consistent estimator of $\sigma^2$ is needed, as in \S \ref{secRandomisation}. Proposition \ref{propPValueYanbo} establishes the error incurred by using the $\chi^2_{n-d_\theta}$ quantiles of the random variable $V:=R^2/\sigma^2$ to approximate those of $\hat V:=R^2/\hat\sigma^2$, on the basis of arbitrary consistent estimator $\hat\sigma^2$, which needs to be assumed independent of $R^2$ for the theoretical derivations. The two estimators $\rcv$ and $\mrcv$ described in \S \ref{secRandomisation} clearly do not satisfy this independence assumption, and it does not seem possible to define an estimator with this property from the same sample. Nevertheless, our numerical investigations found that $\rcv$ and $\mrcv$ computed on a proportion $\gamma=0.6$ of the data, have small enough dependence to yield reasonably good inferential performance, although slightly miscalibrated in small samples.

\begin{proposition}\label{propPValueYanbo}
	Let $\hat\sigma^2$ be a consistent estimator of $\sigma^2$, independent of $R^2$. Then, 
	\[
	\sup_{v>0}|\pr(\hat V \leq v) - \pr( V \leq v)| = o(1).
	\]      
\end{proposition}

\begin{proof}
	The result is a direct application of Polya's theorem. An alternative derivation with possibly very conservative rates is given in Appendix \ref{appAlternativePolya} via a more explicit statement about the rate of convergence of $\hat\sigma^2$.
\end{proof}

As expected in view of the previous observations, we found in simulations that confidence sets of models based on the ancillary component $R^2$ have favourable properties over those based on the co-sufficient pseudo-replicates $\tilde{Q}^{(1)}, \ldots, \tilde{Q}^{(k)}$ if $\sigma^2$ is known, but that calibration is compromised in small samples once $\sigma^2$ is estimated. For sufficiently large sample size, the approach based on the ancillary information with $\sigma^2$ estimated is well calibrated. 

That the test based on $R^2$ is sensitive to estimation of $\sigma^2$ is not unexpected: $R^2$ estimates $\sigma^2$ when the submodel being tested is correct, so it is impossible to completely break the dependence between $R^2$ and $\rcv$ or $\mrcv$, as required in the idealised Proposition \ref{propPValueYanbo}. On the other hand, since $Q$ and $R^2$ are independent, $\tilde{Q}^{(1)}, \ldots, \tilde{Q}^{(k)}$ and $\mrcv$ are expected to be asymptotically independent. This can also be seen on noting that $\EE(W)$ and $\sigma^2$ are orthogonal parameters in the sense of Huzurbazar (1950).

\section{Numerical comparisons}\label{secSim}

\subsection{Experimental design}

This section explores the numerical performance of the approaches based on the co-sufficient and ancillary separations relative to the simpler approach based on sample splitting, which inevitably discards relevant information. We focus on the Gaussian linear model of \S \ref{secExplicitGaussian}. Conditional on all signal variables surviving the reduction phase, all three approaches have the correct theoretical coverage probability for sufficiently large sample size, but the approach based on the ancillary separation is expected to be more sensitive to estimation of the error variance in small samples, as discussed at the end of the previous section. The extent of this effect is uncovered through simulation in idealised settings. The approach based on sample splitting discards information at both the reduction phase and the model assessment phase, and in small samples will therefore discard signal variables with higher probability in the initial phase for an encompassing set $\hat{\mathcal{E}}$ of given size, or conversely, include more variables in $\hat{\mathcal{E}}$ for a fixed probability $\pr(\mathcal{E}^* \subseteq \hat{\mathcal{E}})$. 

In 500 simulation replicates, $n$ independent rows of $X$ were generated from a normal distribution of dimension $p$ with zero mean and covariance matrix $\Sigma$,  where $\Sigma$ is the identity matrix with the top-left $10\times 10$ block replaced by an equi-correlation matrix with correlation $\rho \in (0, 1)$. The outcome $Y$ was generated as $Y\sim N(X\theta, \sigma^2 I_n)$, where $\sigma^2=1$ and $\theta=(t,t,t,0,\ldots,0)^\T$; in other words, the true model is $\mathcal{E}^*=\{1,2,3\}$.  The three signal variables were therefore correlated with each other and with 7 irrelevant variables.

Since the emphasis of this paper is on the model assessment phase, we simplify the reduction phase by choosing $p=400$, so that a single-stage version of Cox reduction can be used rather than the more general two-stage reductions that would be required with a considerably larger number of variables. If $p$ is larger, the reduction phase is more difficult, and the true set of signal variables are less likely in a reduced data set to survive reduction, strengthening the argument for a version that uses all data for both phases of analysis. Detailed analyses of Cox reduction, with extensive numerical comparisons for higher dimensions and different modelling frameworks, are available in Battey and Cox (2018), Hoeltgebaum and Battey (2019) and Lewis and Battey (2025). The latter paper provides insights into how the procedure used at the reduction phase interacts with the model assessment phase. Cox reduction for $p=400$ proceeds by arranging the variable indices at random in a $20\times 20$ grid and traversing the rows and columns, fitting a linear regression to each set of variables. Thus each variable is a covariate in two regressions, alongside different companion variables. Variables that are not significant at a given significance level in either analysis are discarded, the resulting set of variables forming the comprehensive model $\hat{\mathcal{E}}$. In a real application, the significance level in each 20-dimensional regression would ideally be chosen to give an acceptable level of stability for $\hat{\mathcal{E}}$ upon re-randomisation of the variable indices in the grid. For simplicity here, we start with a significance level of $0.05$, reducing it in increments of $0.001$ until the encompassing model $\hat{\mathcal{E}}$ contains at most 15 variables. In very high-dimensional problems, a provisional reduction to 15 variables is likely to be too severe; the number can in principle be any value smaller than $n$, the main difficulty being the computational burden associated with testing a large number of submodels. This, however, is less problematic in a single application than in the many-replicate simulation considered here.

For comparison, we also use for reduction an undertuned lasso, with parameter chosen such that approximately the same number of variables are retained in $\hat{\mathcal{E}}$ as were retained by Cox reduction in each simulation replicate. Specifically, the lasso was implemented with the smallest penalty that ensured selection of at most 15 predictors.

Prior to the model-assessment phase, we estimated the error variance via $\mrcv$ as defined in \eqref{eqMRCV}, as both the co-sufficiency-based and ancillary-based tests require a consistent estimator of $\sigma^2$. This estimator was computed using the first $[0.6 n]$ observations, in order to reduce the dependence between the estimator and the test statistics in the idealised known-variance setting. Estimation with the full sample led to invalid tests under the null, whereas the 60\% version worked well over many scenarios. The conservative variable selection procedure employed for screening in the definition of $\mrcv$ was the undertuned lasso procedure described before. The lasso, having excellent predictive properties, delivers good estimates of variance even when signal variables are discarded in favour of other variables.

For the construction of the confidence set of models $\mathcal{M}$ we compared the following testing procedures, assessing all subsets of $\hat{\mathcal{E}}$ of size 5 or smaller for their compatibility with the data. Those tests not rejected at the $\alpha= 0.05$ significance level constitute $\mathcal{M}$, a $\alpha$-level confidence set of models. 

\begin{itemize}
	\item \textit{Co-sufficient information}: $k$ replicates $\hat{Q}_1,\ldots, \hat{Q}_k$ of $Q$ were computed as described in \S \ref{secRandomisation}, and were tested for uniformity on the hypersphere in $n-d_\theta$ dimensions using the high-dimensional Rayleigh test.
	\item \textit{Ancillary component}: the observed test statistic $\hat{V} = R^2/\mrcv$ discussed in \S \ref{secAncillaryTest} was calibrated against its approximate null distribution $\chi^2_{n-d_\theta}$.
	\item \textit{No adjustment for selection}: for comparison, we compare the above constructions to the na{\"i}ve approach of using the same data for reduction and for model assessment based on a likelihood ratio test of each submodel against the comprehensive model. Since we are in a Gaussian setting, the likelihood-ratio statistic reduces to the $F$ statistic. Under each null model $\mathcal{E}_m$ to be assessed, this $F$ statistic is calibrated against an $F$ distribution on $|\hat{\mathcal{E}}|-|\mathcal{E}_m|$ and $n-|\hat{\mathcal{E}}|$ degrees of freedom, which would be the exact distribution under the null model, had $\hat{\mathcal{E}}$ not been selected in the light of the data.
	\item \textit{Sample splitting}: the same $F$ test is performed as described above, except that the first 60\% of observations are used for reduction and the remaining 40\% for model assessment. This ratio is in the range suggested by Lewis and Battey (2025).
\end{itemize}

We did not attempt the Markov Chain Monte Carlo approaches noted in the introduction, as there were 4943 submodels of $\hat{\mathcal{E}}$ to be assessed in each simulation run when restricting to models of maximal size 5; MCMC methods seem poorly suited for the confidence sets of models problem in view of the computational requirements.

\subsection{Coverage and expected size of the confidence sets}

The quantities of interest are reported in Tables 1--4 for values of $n$, $t$ and $\rho$ each taken at one of two levels in combinations corresponding to a half-replicate of a factorial design: (low, low, low), (high, high, low), (high, low, high), (low, high, high). The reported outcome variables are: the simulated probability that $\mathcal{E}^*$ is retained through the initial reduction; the simulated coverage probability of the confidence set $\mathcal{M}$ at nominal coverage 0.95; and the simulation average of the number of models in the confidence set. Numbers in parentheses indicate empirical standard deviations of the simulation averages.

\begin{table}
	\caption{$n = 100$; $t = 0.5$; $\rho = 0.1$.}
	{
		\begin{tabular}{|l|c|c|c|}
			\hline
			& $\mathbb{P}(\mathcal E^*\subseteq \mathcal{M})$ &   $\mathbb{P}(\mathcal E^*\subseteq \hat{\mathcal{E}})$ & $\mathbb E(\vert \mathcal{M} \vert)$ \\
			\hline
			Cox + co-sufficient ($k = 2$) & 0.94 (0.02) & 0.94 (0.02) & 3097 (147)   \\
			Cox + co-sufficient ($k = 8$) & 0.91 (0.03) & 0.94 (0.02) & 2399 (169) \\
			Cox + ancillary & 0.90 (0.03) & 0.94 (0.02) & 2320 (166) \\
			Cox + $F$ test & 0.02 (0.01) & 0.94 (0.02) & 4 (1)  \\
			Lasso + co-sufficient ($k = 2$) & 0.98 (0.01) & 0.99 (0.01) & 3455 (149) \\
			Lasso + co-sufficient ($k = 8$) & 0.96 (0.02) & 0.99 (0.01) & 2763 (161) \\
			Lasso + ancillary & 0.95 (0.02) & 0.99 (0.01) & 2651 (160)  \\
			Lasso + $F$ test & 0.00 (0.00) & 0.99 (0.01) & 0 (0) \\
			split Cox + split $F$ test & 0.49 (0.05) & 0.51 (0.05) & 1485 (129) \\
			split Lasso + split $F$ test & 0.76 (0.04) & 0.78 (0.04) & 1485 (133) \\
			\hline
		\end{tabular}
		\label{tab1}
	}
\end{table}

%%%%%%%%%%%%%%%%%%%%%%%%%%%%%%%%%%%%%%%%%%%%%%%%%%

\begin{table}
	\caption{$n = 120$; $t = 1$; $\rho = 0.1$.}{
		\centering
		\begin{tabular}{|l|c|c|c|}
			\hline
			& $\mathbb{P}(\mathcal E^*\subseteq \mathcal{M})$
			& $\mathbb{P}(\mathcal E^*\subseteq \hat{\mathcal{E}})$
			& $\mathbb E(\vert \mathcal{M} \vert)$ \\
			\hline
			Cox + co-sufficient ($k = 2$) & 0.96 (0.01) & 1.00 (0.00) & 332 (24) \\
			Cox + co-sufficient ($k = 8$) & 0.88 (0.01) & 1.00 (0.00) & 137 (13) \\
			Cox + ancillary               & 0.86 (0.02) & 1.00 (0.00) & 124 (12) \\
			Cox + $F$ test                & 0.41 (0.02) & 1.00 (0.00) & 31 (1) \\
			Lasso + co-sufficient ($k = 2$) & 0.96 (0.01) & 1.00 (0.00) & 402 (29) \\
			Lasso + co-sufficient ($k = 8$) & 0.88 (0.01) & 1.00 (0.00) & 158 (14) \\
			Lasso + ancillary               & 0.86 (0.02) & 1.00 (0.00) & 146 (13) \\
			Lasso + $F$ test                & 0.00 (0.00) & 1.00 (0.00) & 0 (0) \\
			split Cox + split $F$ test      & 0.94 (0.01) & 0.99 (0.00) & 77 (3) \\
			split Lasso + split $F$ test    & 0.95 (0.01) & 1.00 (0.00) & 84 (3) \\
			\hline
		\end{tabular}
		\label{tab_sim_new4}
	}
\end{table}

%%%%%%%%%%%%%%%%%%%%%%%%%%%%%%%%%%%%%%%%%%%%%%%%%%

\begin{table}
	\caption{$n = 120$; $t = 0.5$; $\rho = 0.5$.}{
		\centering
		\begin{tabular}{|l|c|c|c|}
			\hline
			& $\mathbb{P}(\mathcal E^*\subseteq \mathcal{M})$
			& $\mathbb{P}(\mathcal E^*\subseteq \hat{\mathcal{E}})$
			& $\mathbb E(\vert \mathcal{M} \vert)$ \\
			\hline
			Cox + co-sufficient ($k = 2$) & 0.97 (0.01) & 1.00 (0.00) & 2863 (67) \\
			Cox + co-sufficient ($k = 8$) & 0.92 (0.01) & 1.00 (0.00) & 1872 (65) \\
			Cox + ancillary               & 0.90 (0.01) & 1.00 (0.00) & 1747 (64) \\
			Cox + $F$ test                & 0.37 (0.02) & 1.00 (0.00) & 54 (3) \\
			Lasso + co-sufficient ($k = 2$) & 0.98 (0.01) & 1.00 (0.00) & 2597 (55) \\
			Lasso + co-sufficient ($k = 8$) & 0.91 (0.01) & 1.00 (0.00) & 1841 (54) \\
			Lasso + ancillary               & 0.90 (0.01) & 1.00 (0.00) & 1745 (53) \\
			Lasso + $F$ test                & 0.00 (0.00) & 1.00 (0.00) & 0 (0) \\
			split Cox + split $F$ test      & 0.93 (0.01) & 0.98 (0.01) & 1200 (44) \\
			split Lasso + split $F$ test    & 0.91 (0.01) & 0.95 (0.01) & 1121 (37) \\
			\hline
		\end{tabular}
		\label{tab_sim_new3}
	}
\end{table}

%%%%%%%%%%%%%%%%%%%%%%%%%%%%%%%%%%%%%%%%%%%%%%%%%%

\begin{table}
	\caption{$n = 100$; $t = 1$; $\rho = 0.5$.}{
		\begin{tabular}{|l|c|c|c|}
			\hline
			& $\mathbb{P}(\mathcal E^*\subseteq \mathcal{M})$
			& $\mathbb{P}(\mathcal E^*\subseteq \hat{\mathcal{E}})$
			& $\mathbb E(\vert \mathcal{M} \vert)$ \\
			\hline
			Cox + co-sufficient ($k = 2$) & 0.94 (0.01) & 1.00 (0.00) & 620 (29) \\
			Cox + co-sufficient ($k = 8$) & 0.86 (0.02) & 1.00 (0.00) & 228 (13) \\
			Cox + ancillary               & 0.83 (0.02) & 1.00 (0.00) & 200 (11) \\
			Cox + $F$ test                & 0.74 (0.02) & 1.00 (0.00) & 53 (1) \\
			Lasso + co-sufficient ($k = 2$) & 0.95 (0.01) & 1.00 (0.00) & 642 (27) \\
			Lasso + co-sufficient ($k = 8$) & 0.85 (0.02) & 1.00 (0.00) & 305 (16) \\
			Lasso + ancillary               & 0.83 (0.02) & 1.00 (0.00) & 265 (14) \\
			Lasso + $F$ test                & 0.00 (0.00) & 1.00 (0.00) & 0 (0) \\
			split Cox + split $F$ test      & 0.96 (0.01) & 1.00 (0.00) & 248 (11) \\
			split Lasso + split $F$ test    & 0.96 (0.01) & 1.00 (0.00) & 237 (10) \\
			\hline
		\end{tabular}
		\label{tab4}
	}
\end{table}

The half-replicate of Tables 1--4 allows estimation of the main effects of increasing each of $n$, $t$ and $\rho$ from its low to its high level. These are presented in Table 5, where the reported effect on the coverage probability is multiplicative on the odds scale, with values close to one indicative of a null effect, and the reported effect on the expected size is additive on the log scale, with values close to zero suggesting a null effect.

%%%%%%%%%%%%%%%%%%%%%%%%%%%%%%%%%%%%%%%%%%%%%%%%%%
\begin{table}
	\caption{Point estimates of effects of $n$, $t$ and $\rho$ from the half-replicate of the factorial design (Tables 1--4). Reported effects are multiplicative on the odds scale for $\mathbb{P}(\mathcal E^*\subseteq \mathcal{M})$ and additive on the log scale for $\mathbb E(\vert \mathcal{M} \vert)$.}
	{  
		\begin{tabular}{|l|c|c|c|c|c|c|}
			\hline
			\multirow{2}{*}{}     
			& \multicolumn{3}{|c|}{$\mathbb{P}(\mathcal E^*\subseteq \mathcal{M})$} 
			& \multicolumn{3}{|c|}{$\mathbb E(\vert \mathcal{M} \vert)$} \\
			\cline{2-7}
			& $n$ & $t$ & $\rho$ & $n$ & $t$ & $\rho$ \\
			\hline
			Cox + co-sufficient ($k = 2$) & 1.334 & 0.929 & 1.077 & -0.176 & -1.250 & 0.005 \\
			Cox + co-sufficient ($k = 8$) & 1.080 & 0.789 & 0.988 & -0.064 & -1.133 & -0.052 \\
			Cox + ancillary & 1.059 & 0.780 & 0.945 & -0.076 & -1.116 & -0.062 \\
			Cox + $F$ test & 1.628 & 3.582 & 3.294 & 0.136 & 0.140 & 0.311 \\
			Lasso + co-sufficient ($k = 2$) & 1.061 & 0.660 & 0.943 & -0.088 & -1.264 & -0.053 \\
			Lasso + co-sufficient ($k = 8$) & 0.859 & 0.643 & 0.755 & -0.110 & -1.184 & -0.089 \\
			Lasso + ancillary & 0.878 & 0.647 & 0.784 & -0.117 & -1.164 & -0.097 \\
			Lasso + $F$ test & -- & -- & -- & -- & -- & -- \\
			split Cox + split $F$ test & 1.733 & 2.330 & 2.147 & -0.157 & -0.776 & -0.040 \\
			split Lasso + split $F$ test & 1.261 & 1.943 & 1.418 & -0.171 & -0.735 & -0.071 \\
			\hline
		\end{tabular}
		\label{tab6}
	}
\end{table}

The tables reveal several aspects that were anticipated qualitatively. Use of the $F$ test without sample splitting leads to extremely poor calibration of the confidence set of models, which is restored through sample splitting provided that the true model survives the reduction phase, this being more likely when $n$ and $t$ are large. The role of $\rho$ in the probability of survival is different in the two split-sample versions, high $\rho$ being beneficial for Cox reduction and detrimental for the lasso. Conditional on survival, the co-sufficient and ancillary tests are more conservative in settings where $t$ is low or $\rho$ is high. Additional simulations, not reported, found that the co-sufficient and ancillary tests were exactly calibrated at their nominal levels across all settings when the sample size was sufficiently large or when the true error variance was used instead of $\mrcv$, suggesting that the small-sample over-coverage in low-information settings is due to variance estimation. In high-information settings, taking $k$ as large as 20 gives coverage close to nominal, and produces smaller confidence sets than the co-sufficient test with smaller values of $k$, verifying intuition. These high-information settings are, however, cases where the split $F$-test dominates the rest. 

Simulations (not reported) with the distribution of the error term taken as student-$t$ gave qualitatively similar conclusions. In that case, the distribution theory for the oracle version of the co-sufficient test remains valid, but that for the $F$ test is invalidated. The survival probabilities of the full set of signal variables are lowered in this case, and the distribution theory for $\mrcv$ established in Proposition \ref{propMRCV} is no longer valid, affecting Theorem \ref{thmQ}.

In the Gaussian context, because the $F$ test is exact, the coverage probability is only affected by the reduced survival probability in the first stage and not by the small sample size at the second stage, and the efficacy of sample splitting is high, the directionality of the $F$ test compensating to a large extent for the smaller sample size, provided that all signal variables survive reduction. In broad terms, therefore, the main aspect influencing the relative advantage of the co-sufficient and ancillary tests over the split-sample tests is the difficulty of the reduction phase. In non-linear regression models, where a likelihood ratio test has to be used in place of the $F$ test, the coverage probability of $\mathcal{M}$ is also affected by the smaller sample size in the second stage, due to the $\chi^2$ approximation to the likelihood-ratio statistic, although this too could be improved by higher-order correction. We discuss briefly in \S \ref{secExtension} the extension of co-sufficient model assessment to settings other than the Gaussian linear model.

Between the two values of $k$ reported, intermediate choices of $k$ interpolate between the reported values of Tables 1--5 in terms of coverage probability and the expected size of the confidence set, with small $k$ achieving higher coverage at the expense of larger confidence sets.

The strategy for reduction is not the focus of this paper. Nevertheless, the tables reflect aspects of the discussion of Lewis and Battey (2025). In particular, Cox reduction leads, on average, to smaller confidence sets than the lasso for the same coverage, as the covariates retained by Cox reduction typically span a larger space than those retained by the lasso. As a result, the retained covariates from Cox reduction fit the data better, making it harder for small sub-models to pass the $F$-test. Cox reduction is also relatively better at retaining signal variables through the reduction phase when the correlation between signal and noise variables is high, and vice versa, as the lasso is more likely to select noise variables that are highly correlated with signal variables. 

In all correctly calibrated cases, we see that the size of the confidence set of models is rather large, particularly for the co-sufficient test. Small errors of calibration seem worth the appreciable reduction in the number of models retained from splitting the sample. The large errors of coverage, as reported in Table 1, are more severe, making the co-sufficient approach more relevant at small sample sizes. The maximal size of the confidence set obtained by assessing models up to size 5 from a reduced set with $|\hat{\mathcal{E}}|=15$ is 4943. Thus, all approaches have power to reject false models in these low-information settings. Cox and Battey (2017) discuss how to extract compact messages from large confidence sets of models.

\section{Extensions to other settings}\label{secExtension}

\subsection{Discussion and immediate extension}

The ideas of \S \ref{secCosuff} and \S \ref{secAncill} extend in principle to any parametric regression model admitting a sufficient, or approximately sufficient, statistic of dimension smaller than $n$. An exact or approximate analytic approach for non-Gaussian models is likely to lead to more marked improvements over sample splitting relative to the Gaussian case, as the $F$ test for model adequacy  is exact, conditional on the true model $\mathcal{E}^*$ having survived the first stage, while the log-likelihood test needed for more general models is based on an asymptotic approximation.

The details of implementation for non-Gaussian models need to be worked out on a case by case basis and may well involve considerable methodological innovation beyond the ideas presented in \S \ref{secOperational}. An exception is the use of the log-normal distribution for modelling positive outcomes, in which case the ideas of \S \ref{secOperational} apply directly after logarithmic transformation. The log-normal distribution is one of several popular parametric distributions for survival outcomes with covariate-dependence usually modelled in traditional accelerated-life form. 

The extension to other models is less direct and it remains unclear whether a randomisation approach paralleling \S \ref{secRandomisation} would be needed, as its necessity in the Gaussian case arose from rotational invariance of the Haar measure. Since the notional co-sufficient statistic is not uniformly distributed in any of the examples of the present section, a degree of power would be expected without any synthetic replication. Further work is needed to understand these cases.

In \S \ref{secExpFam} and \S \ref{secAL} we emphasise the structure in canonical exponential-family regression models and in accelerated life models that simplifies the required calculations. 

\subsection{Canonical exponential family regression models}\label{secExpFam}

An important role is played by canonical exponential family regression models for outcomes $Y=(Y_{1},\ldots,Y_n)$, assumed conditionally independent given observations on covariates $x$. 
Assuming the model is correctly specified, the conditional probability density or mass function of $Y=(Y_{1},\ldots,Y_n)$ at $y$ is
\[
f(y;x_{1}^\T\theta, \ldots, x_n^\T\theta)=\exp\Bigl[\phi^{-1}\Bigl\{\theta^{T}\sum_{i=1}^{n}x_i y_i - \sum_{i=1}^n K(x_i^{T}\theta)\Bigr\}\Bigr]\prod_{i=1}^n h(y_i, \phi^{-1}),
\]
where $\phi$ is a dispersion parameter, always equal to 1 in some models. Example cases are $K(\eta)=\eta^2/2$ for the normal theory linear model, $K(\eta)=e^\eta$ for Poisson regression, and $K(\eta)=\log(1+e^\eta)$ for binary regression.

Provided that any floating dispersion parameter is known, the sufficient statistic for $\theta$ is $S = \sum_{i=1}^n x_i Y_i = X^\T Y$, where each $x_i$ is of the same dimension $d=d_{\theta}$ as $\theta$. There is a reduction by sufficiency from $n$ to $d$ and no ancillary statistic in the sense of \S \ref{secAncill}. The insights of \S \ref{secTangent} show that the co-sufficient manifold $\cosuff(s^o)$ is flat. We consider in more detail the example of logistic regression.

\subsection{Logistic regression}\label{secLogistic}

Unlike the Gaussian case of \S \ref{secGL}, where the conditional distribution reduces to a marginal distribution, the co-sufficient statistic $Q(s^o)$ in logistic regression is not observable and is defined only implicitly through \eqref{eqCosuffStat}. The conditional distribution of interest is
\begin{equation}\label{eqConditionalLogistic}
	\pr(Y=y \mid S=s^o) = \frac{\pr(Y=y)}{\sum_{y'\in \cosuff(s^o)}\pr(Y=y')}.
\end{equation}
With a different objective in mind, Hanley and Roy (2025) obtained, by linear integer programming, all sets of binary outcomes $y\in\{0,1\}^n$ compatible with the observed value of the sufficient statistic $s^o=X^\T y^o$, i.e.~a complete enumeration of $\cosuff(s^o)$, which is a finite set in this case. Their original motivation for obtaining these solutions was to re-analyse by a superficially more modern approach, an old analysis due to Penrose (1933), aided by Fisher. The conclusion of their pedagogical piece was that all solutions to this linear integer programming problem gave the same inference for the logistic regression parameters as that obtained by Penrose and Fisher, as they must by construction, by the definition of sufficiency. Summation over Hanley's and Roy's (2025) full solution set in equation \eqref{eqConditionalLogistic} provides a means of calculating the conditional distribution of interest. In practice, the probabilities on the right hand side of \eqref{eqConditionalLogistic} depend on the unknown parameter vector $\theta$, which can be replaced by the maximum likelihood solution $\hat\theta$.

\subsection{Accelerated life models}\label{secAL}

Let $T_0$ be a positive random variable representing the outcome of individuals with covariates at baseline $z=0$. Here we use $z$ in place of $x$, as $x$ is defined to be one dimension larger and includes a column of ones. The density function transforms under scaling as
\[
f_{\sigma T_0}(t)dt = \sigma^{-1}f_{T_0}(t/\sigma)dt,
\]
implying that $T=T_0/\rho(z)$ has a density function of accelerated life form
\[
f_T(t;z^\T\beta)=\rho(z) f_{T_0}(\rho(z)t),
\]
where the effect of the covariates is to multiply the baseline rate $\rho_0$ by $\rho(z)$, or equivalently to multiply the baseline scale $\sigma_0$ by $\sigma(z)$. With $\rho(z)=\exp(z^\T\beta)$,
\begin{equation}\label{eqLogT}
	Y:=\log T = \log T_0 - \log \rho(z) = \log T_0 - z^\T\beta = \mu_0 - z^\T\beta +\varepsilon,
\end{equation}
where $\mu_0 = \EE(\log T_0)$ and $\varepsilon$ is a mean-zero error term, whose distribution is determined by that of the baseline variable $T_0$. The broad ideas of \S \ref{secGL} apply to the log-transformed outcome model \eqref{eqLogT} with $\theta$ one dimension larger than $\beta$ to incorporate the intercept $\mu_0$. If the baseline outcome distribution is log-normally distributed, no modification is needed. Here we consider an alternative choice of outcome distribution to highlight the considerations involved.

Suppose for concreteness that $T_0$ has a log-logistic distribution of rate $\rho_0>0$ and scale $\tau>0$, then $Y_0 = \log T_0$ has a logistic distribution of mean $\log \rho_0$ and variance $\tau^2$ so that $Y$ satisfies \eqref{eqLogT} with $\varepsilon$ having density function 
\begin{equation}\label{eqLogisticDens}
	f_{\varepsilon}(\epsilon)=\frac{\exp(-\epsilon/\tau)}{\tau(1+\exp(-\epsilon/\tau))},
\end{equation}

With a slight ambiguity in notation, resolved by context, $\varepsilon$ denotes both the random variable with density function \eqref{eqLogisticDens} and the vector of independent replicates following the same distribution. These independent replicates are written $\varepsilon_{1},\ldots, \varepsilon_{n}$. Let $X$ be the $n\times d_\theta$-dimensional matrix of covariate information. Following \S\ref{secExplicitGaussian}, take $U = (u_{ij})\in \mathcal{X}^\perp$ to be the $n\times (n-d_\theta)$ matrix of orthonormal eigenvectors corresponding to the unit eigenvalues of $M=I_n - X(X^\T X)^{-1}X^\T$ and let $W=U^\T Y$. Under the null model, $\mu=X\theta \in \mathcal{X}$, and the distribution of $W$ does not depend on $\theta$. 

We only consider inference based on $W$, rather than on $Q=U^\T Y/\|U^\T Y\|$. The latter has a distribution on the unit hypersphere in $\mathcal{X}^\perp$, however this distribution is not uniform and is more difficult to calculate in the present setting. The characteristic function $\varphi_W(z)$ of $W$ can be written in terms of the characteristic function $\varphi_\varepsilon$ of $\varepsilon$ as $\varphi_{W}(z) = \varphi_\varepsilon(U^\T z)$ where
\begin{equation}\label{eqCF}
	\varphi_\varepsilon(t) = \prod_{i=1}^n \varphi_{\varepsilon_i}(t_i) = (\pi \tau)^n \prod_{i=1}^n \frac{t_i}{\sinh(\pi\tau t_i)}.
\end{equation}
Inversion of $\varphi_{W}(z)$ gives the density function 
\begin{equation}\label{eqVoldoire}
	f_{W}(w)=\frac{1}{(2\pi)^{(n-d_\theta)}} \int_{[-\pi, \pi]} \cdots \int_{[-\pi, \pi]} \Re\{\exp(-i z^\T w) \varphi_\varepsilon(U^\T z)\} dz_1\cdots dz_{n-d_\theta},
\end{equation}
where $\Re(\cdot)$ denotes the real part of a complex number. Voldoire et al.~(2024) suggest a procedure for obtaining samples from \eqref{eqVoldoire}. By comparing observed values of $W$ to the tails of this sampling-based distribution, adequacy of the model can be assessed. This requires the logistic standard deviation $\tau$ in \eqref{eqCF} to be replaced by a consistent estimate, as in \S \ref{secRandomisation}. Ideally, \eqref{eqVoldoire} would admit a tractable analytic approximation, as direct sampling is computationally prohibitive for the reasons noted earlier.

\subsection{General non-exponential family models}\label{secTangent}

The $n\times d$ matrix of gradient vectors
\begin{equation}\label{eqV}
	V(y^o)=\left.\frac{\partial s^\T(y)}{\partial y}\right|_{y=y^o}
\end{equation}
specifies the normal directions to $\cosuff(s^o)$ at $y^o$. The space spanned by the columns of $V(y^o)$ is the normal subspace of $\cosuff(s^o)$ at $y^o$ and the orthogonal complement
\[
V^\perp(y^o)=\{u\in\RR^n: u^\T V(y^o) = 0\in \RR^d\}
\] 
defines the tangent plane at $y^o$. For canonical exponential-family regression models without a floating dispersion parameter, $V(y^o)$ does not depend on $y^o$, implying that the co-sufficient manifold is flat. This idea underpins the tangent exponential model of Fraser (1988, 1990) and Fraser and Reid (1988, 1995), whose objective was higher order conditional inference on $\theta$. Our problem is in essence the opposite of theirs. The establishment of general analytic approximations is an open problem. In principle, such approximations could replace the importance-sampling based approach suggested for the log-logistic accelerated life model of the previous section, and possibly also the integer-programming based approach for the logistic regression formulation of \S \ref{secLogistic}.

\section{Closing discussion}\label{secClosing}

The literature on post-selection inference has focused on the conceptually different problem of performing inference on a parameter following a data-based selection of effects on which formal inference is sought. The problem of post-reduction inference for confidence sets of models raises similar technical challenges, but is conceptually distinct. In particular, the object of inference in the confidence sets of models problem is unambiguously defined, and the two phases of analysis are such that there is a clearer separation of information in the data for the different phases. This contrasts strongly with the more widely studied post-selection inference problems, which address almost the same scientific question in different guises: a selection phase asking whether the effect is zero, and an inference phase asking what values of strongly-suggested effects are compatible with the data. 

In the confidence sets of models formulation, the most relevant hypotheses for test, as small submodels of $\hat{\mathcal{E}}$, are stable over hypothetical replication from the same generating process. Provided that $\hat{\mathcal{E}}$ reliably includes $\mathcal{E}^*$, the double use of the data is only problematic if the overfitted $\hat{\mathcal{E}}$ is used as a basis for comparison of fit, as in the likelihood ratio test. The present paper showed how to avoid this comparison by operationalising the appropriate co-sufficiency and ancillary separations. If the ancillary statistic for the interest component depends weakly on nuisance parameters as in the weaker definitions of ancillarity on page 38 of Barndorff-Nielsen and Cox (1994), then the small-sample coverage properties of the ancillary-based assessment may be compromised through estimation of these nuisance parameters. 
Numerical work showed that, although the efficiency of sample splitting is relatively high for the Gaussian model under consideration here due to the exactness of the second-stage test, there are modest advantages of co-sufficient information partitioning in settings where the reduction phase is difficult, for instance when the sample size is small relative to dimension or the signal strength is low.

If the modelling assumptions are accepted, the exact analytic calculations of the present paper provide insight, reduce approximation errors, and considerably reduce computation relative to approaches based on Monte Carlo sampling.

A natural first reaction to the tables in \S \ref{secSim} is that the number of models in the confidence set is inconveniently large. This is an honest reflection of the limitations of the data, and emphasises the dangers of reporting a single model when many are compatible with the data at the chosen significance level. Statistical analysis can go no further, and any choice between models in the confidence set requires either additional data or subject-matter expertise. Cox and Battey (2017) discuss how one might extract compact messages from large confidence sets of models. Specifically, among the large number of models reported in their supplementary material, two variables, call them $v_1$ and $v_2$, are present in 96\% and 94\% of cases. In 78\% of the models in which $v_2$ is not present, another variable, $v_3$, is present in its place, and only 1\% of models include neither $v_2$ nor $v_3$.

If prediction is the sole objective, an arbitrary well-fitting model, such as that selected by the lasso, will typically suffice, and a confidence set of models is unnecessary. Confidence sets may nevertheless be used to form confidence distributions for predictions, as discussed in Section S3 of the supplement of Lewis and Battey (2025).

\section*{Acknowledgements}

HB is grateful to James Hanley for sharing a preprint of Hanley and Roy (2025). We thank two reviewers and the associate editor for detailed reading, thoughtful suggestions, and for indicating an oversight in the proof of Proposition 2. EPSRC fellowship EP/T01864X/1 (to HB) is gratefully acknowledged.

\section*{Supplementary material}

The supplementary material contains Appendices A and B referenced in the main text. These include proofs and technical details.

\newpage

{\centering

SUPPLEMENTARY MATERIAL

}

 \begin{appendix}

 	\section{Power and na{\"ive} data-partitioning}\label{secPartitioning}

From $k$ disjoint subsamples of size $n_k=n/k$, we consider initially $Q^{(1)},\ldots,Q^{(k)}$, constructed as above but based on disjoint subsamples $(Y^{(1)},X^{(1)}),\ldots, (Y^{(k)},X^{(k)})$. 

The resulting co-sufficient manifolds for each subsample, and the corresponding null distributions for $Q^{(1)},\ldots,Q^{(k)}$, while comparable when viewed as a hypersphere embedded in $\RR^{n_k - d_\theta}$, are not comparable when viewed within the larger ambient space $\RR^{n_k}$, as the column spaces of $X^{(1)},\ldots, X^{(k)}$
define $k$ different partitions $\mathcal{X}^{(i)}\oplus\mathcal{X}^{(i)\perp}=\RR^{n_k}$ and each co-sufficient manifold is embedded within a different subspace $\mathcal{X}^{(i)\perp}$.

To expose the nature of the difficulties in using these partitioned data to assess compatibility with $\mathcal{E}_m$, suppose for some preliminary insight that the null model $\mu=X\theta$ is violated, the true direction of the mean vector being $\mu=X\theta + Z\lambda$. For notational convenience and to avoid repetition, we state the following calculations for a single sample of size $n$, to be reinterpreted later in the context of the subsamples of size $n_k$. With $U$, $W$ and $d_\theta$ as defined in \S\ref{secExplicitGaussian}, the density function of $W$ at $w$ is
\begin{equation}\label{eqWAlt}
	%f_{W}(w)dw=
	\frac{\exp\{-(a+w^\T w)/2\sigma^2\}}{(2\pi\sigma^2)^{(n-d_\theta)/2}}\exp\Bigl(-\frac{w^\T U^\T Z \lambda}{\sigma^2}\Bigr)dw,
\end{equation}
where $a=\lambda^\T Z^\T U U^\T Z \lambda$. The transformation $w\mapsto (w/\|w\|,w^\T w) = (q, r^2)$ gives the joint density function of $Q$ and $R^2$ as
\[
\frac{(r^2)^{(n-d_\theta)/2 -1}}{2(2\pi\sigma^2)^{(n-d_\theta)/2}} \exp\Bigl(-\frac{a+r^2}{2\sigma^2}\Bigr)\exp\Bigl(-\frac{(r^2)^{1/2}q^\T U^\T Z\lambda}{\sigma^2}\Bigr)dr^2 d \text{Vol}^{n-d}(q).
\]
In the absence of a closed form for the marginal distribution of $Q$, Appendix \S \ref{appQNonNull} establishes the approximation
\begin{equation}\label{eqQNonNull}
	f_Q(q)d\text{Vol}^{n-d}(q)  \simeq 
	\frac{\Gamma\bigl(\frac{n-d_\theta}{2}\bigr)}{2\pi^{(n-d_\theta)/2}}\exp\Bigl(-\frac{a}{2\sigma^2}\Bigr) \exp\biggl\{\frac{(n-d_\theta) q^\T U^\T Z\lambda}{\sigma\sqrt{2}}\biggr\}d \text{Vol}^{n-d}(q), 
\end{equation}
which we write as
\[
f_Q(q)d\text{Vol}^{n-d}(q) \simeq \frac{\Gamma\bigl(\frac{n-d_\theta}{2}\bigr)}{2\pi^{(n-d_\theta)/2}} \cdot h(q)d \text{Vol}^{n-d}(q).
\]
Approximation \eqref{eqQNonNull}, although obtained under the notional limiting operation $(n-d_\theta)\rightarrow \infty$, is essentially exact, the only error coming from Stirling's approximation.

The first multiplicative factor on the right hand side of \eqref{eqQNonNull} is the density function of $Q$ under the null hypothesis. Thus, power is low when $h(q)$ is close to 1 for all $q$.  This occurs when $Z\lambda$ is close to zero or when $Z$ lies close to the column span of $X$. The conclusion agrees qualitatively with the power analysis for the likelihood-ratio test in the absence of any dependencies from reduction, as derived in \S 4.2 of Lewis and Battey (2025) using a limiting argument in Le Cam-type contiguous alternatives. 

Consider now a test of uniformity on the unit hypersphere in $\RR^{n_{k}-d_\theta}$ based on the partitioned data, discarding the orientation of each $\mathcal{X}^{(i)\perp}$ within $\RR^{n_k}$; $k$ now plays the role of sample size for the null uniform samples. Equation \eqref{eqQNonNull} suggests that the non-null distribution of $Q^{(i)}$ is well approximated by a von Mises-Fisher distribution with mean direction $m^{(i)} = U^{(i)T}Z^{(i)}\lambda$ and dispersion parameter $\kappa^{(i)} = \sigma^{-1} 2^{-1/2} (n_k - d_\theta) \Vert U^{(i)T}Z^{(i)}\lambda \Vert$. Since the mean vectors $m^{(i)}$ depend on how the data are partitioned, and because these themselves are in general numerically indistinguishable from uniformly distributed random vectors on the unit hypersphere in $\RR^{n_k-d_\theta}$ when standardised to unit length, the distribution of the test statistic when the assumed model is violated is hard to distinguish from that under the model postulated under the null hypothesis. Put differently, by ignoring the orientation of the embedding manifold $\mathcal{X}^{(i)\perp}$, the analysis is made too unconditional.

We propose, instead, the following procedure, which implicitly conditions on the observed column space of $X$, ensuring that pseudo-replicates of $Q$ are constrained to the same co-sufficient manifold as $Q$ when viewed as embedded within $\RR^{n-d_\theta}$.

\section{Proofs and derivations}\label{appDerivations}

\subsection{Derivation of equation \eqref{eqQNonNull}}\label{appQNonNull}

From the joint density function
\[
\frac{(r^2)^{(n-d_\theta)/2 -1}}{2(2\pi\sigma^2)^{(n-d_\theta)/2}} \exp\Bigl(-\frac{a+r^2}{2\sigma^2}\Bigr)\exp\Bigl(-\frac{(r^2)^{1/2}q^\T U^\T Z\lambda}{\sigma^2}\Bigr)dr^2 d \text{Vol}^{n-d}(q)
\]
of $Q$ and $R^2$, expand the exponential function around zero in $\exp(-(r^2)^{1/2}q^\T U^\T Z\lambda/\sigma^2)$. This leads to an infinite-series representation of the marginal density function of $Q$:
\begin{eqnarray*}
	\frac{\exp(-a/2\sigma^2)}{2(2\pi\sigma^2)^{(n-d_\theta)/2}} \sum_{m=0}^{\infty}\frac{1}{m!}\biggl(\frac{q^\T U^\T Z\lambda}{\sigma^2} \biggr)^m \left(\int_{0}^{\infty}(r^2)^{\tfrac{n-d_\theta+m}{2} - 1}\exp(-r^2/2\sigma^2)dr^2\right)d \text{Vol}^{n-d}(q).
\end{eqnarray*}
After a change of variables from $r^2$ to $r^2/2\sigma^2$, this is solved as the infinite series of gamma integrals
\begin{equation}\label{eqSeries}
	f_Q(q)d \text{Vol}^{n-d}(q)=\frac{\exp(-a/2\sigma^2)}{2\pi^{(n-d_\theta)/2}}\sum_{m=0}^{\infty}\frac{1}{m!}\Bigl(\frac{\sqrt{2} q^\T U^\T Z\lambda}{\sigma} \Bigr)^m \Gamma\Bigl(\frac{n-d_\theta+m}{2}\Bigr)d \text{Vol}^{n-d}(q),
\end{equation}
which is shown to be convergent by consideration of the ratios of successive terms:
\[
M_m=\frac{m!\bigl(\frac{\sqrt{2}q^\T U^\T Z\lambda}{\sigma} \bigr)^{m+1} \Gamma\bigl(\frac{n-d_\theta+(m+1)}{2}\bigr)}{(m+1)!\bigl(\frac{\sqrt{2} q^\T U^\T Z\lambda}{\sigma} \bigr)^m \Gamma\bigl(\frac{n-d_\theta+m}{2}\bigr)} = \Bigl(\frac{\sqrt{2}q^\T U^\T Z\lambda}{(m+1)\sigma}\Bigr)\frac{\Gamma\bigl(\frac{n-d_\theta+(m+1)}{2}\bigr)}{\Gamma\bigl(\frac{n-d_\theta+m}{2}\bigr)}.
\]
On using Stirling's asymptotic approximation to the gamma integral in the form $\Gamma(x+\alpha)/\Gamma(x)\simeq x^\alpha$ for large $x$ and $\alpha$ fixed we see that $\lim_{m\rightarrow \infty}M_m=0$, so that the series representation for $f_{Q}(q)d \text{Vol}^{n-d}(q)$ is convergent. On multiplying and dividing \eqref{eqSeries} by $\Gamma\{(n-d_\theta)/2\}$ and approximating the resulting gamma ratios using Stirling's approximation as above, we obtain the approximation
\begin{eqnarray*}
	\nonumber	f_Q(q)d \text{Vol}^{n-d}(q) &\simeq  & \frac{\exp(-a/2\sigma^2)}{A}\sum_{m=0}^{\infty}\frac{1}{m!}\Bigl(\frac{\sqrt{2} q^\T U^\T Z\lambda}{\sigma} \Bigr)^m \Bigl(\frac{n-d_\theta}{2}\Bigr)^m  d \text{Vol}^{n-d}(q) \\
	&=& \frac{\exp(-a/2\sigma^2)}{A}\exp\Bigl(\frac{(n-d_\theta) q^\T U^\T Z\lambda}{\sigma\sqrt{2}}\Bigr)d \text{Vol}^{n-d}(q),
\end{eqnarray*}
where, on recalling that $d=d_\theta +1$, 
\[
A=\frac{2\pi^{(n-d_\theta)/2}}{\Gamma\bigl(\frac{n-d_\theta}{2}\bigr)} = \int_{\mathcal{S}_{n-d_\theta - 1}} d \text{Vol}^{n-d_\theta - 1}(q)
\]
is the surface area of the unit hypersphere $\mathcal{S}_{n-d_\theta - 1}$ embedded in $\RR^{n-d_\theta}$. This is expression \eqref{eqQNonNull}.

\subsection{Proof of Proposition \ref{lemmaRand}}\label{appLemmaRand}

\begin{proof}
	The $i$-th replicate of $Y$ can be written as $\tilde Y^{(i)} = Y - \sum_{j = 1}^{i-1} b_j L_j + a_iL_i$, where $L_i$ is the $i$-th column of $L$, $a_k = 0$, and empty sums of the form $\sum_{k = 1}^0$ are defined to be 0. We need to verify that all replicates have covariance $k\sigma^2I_n$ and are independent. We have:
	\begin{equation*}
		\mathrm{Cov}(\tilde Y^{(i)}) = \Bigl(\sigma^2 + \sum_{j = 1}^{i-1} b_j^2 + a_i^2\Bigr)I_n = \Bigl(\sigma^2 + \sum_{j = 1}^{i-1} b_j^2 + \sigma^2 (k-1) - \sum_{j = 1}^{i-1} b_j^2\Bigr)I_n = k\sigma^2I_n.
	\end{equation*}
	For every $1\leq i_1 < i_2\leq k$,
	\begin{equation*}
		\mathrm{Cov}(\tilde Y^{(i)}, \tilde Y^{(j)}) = \Bigl(\sigma^2 + \sum_{j = 1}^{i_1 - 1}b_j^2 - a_{i_1}b_{i_1}  \Bigr)I_n = \Bigl(\sigma^2 + \sum_{j = 1}^{i_1 - 1}b_j^2 - \sigma^2 -   \sum_{j = 1}^{i_1 - 1}b_j^2 \Bigr)I_n = 0_{n\times n},
	\end{equation*}
	where $0_{n\times n}$ is a matrix of 0's and we are taking $b_0 = 0$ if $i_1 = 1$.
\end{proof}

\subsection{Proof of Proposition \ref{propMRCV}}

\begin{proof}
	For $j\in \{1, 2\}$, let $e_j = \vert \hat{\mathcal{E}}^{(j)} \vert$ and $\mathcal{A}_j = \{\mathcal{E}\subseteq \{1, \ldots, p\} \colon \mathcal{E}^*\subseteq \mathcal{E}, \vert \mathcal{E}\vert = e_j \}$. Write 
	\begin{align}\label{eq:Estimator}
		\hat\sigma^2 = \frac{\Vert \hat M_1 Y^{(1)} \Vert^2 + \Vert \hat M_2 Y^{(2)} \Vert^2}{\nu},
	\end{align}
	where $\hat M_j$ projects onto the null-space of $X^{(j)}_{\hat{\mathcal{E}}^{(-j)}}$. Consider an oracle estimator 
	\begin{align}\label{eq:oracle-estimator}
		\bar\sigma^2 = \frac{\Vert \bar M_1 Y^{(1)} \Vert^2 + \Vert \bar M_2 Y^{(2)} \Vert^2}{\nu},
	\end{align}
	where $\bar M_j$ projects onto the null-space of $X^{(j)}_{\bar E^{(-j)}}$ for an arbitrary, non-random $\bar E^{(-j)}\in \mathcal A_-j$. Since the encompassing sets are deterministic and contain the true model, we know that
	\begin{align}\label{eq:chi-square}
		\bar\sigma^2 \stackrel{d}{=} \sigma^2 \frac{\chi^2_\nu}{\nu}. 
	\end{align}
	Let $\mathbf{1}_C$ be the indicator of the containment event, $C = \{\mathcal{E}^*\subseteq \hat{\mathcal{E}}^{(1)} \cap \hat{\mathcal{E}}^{(2)}\}$. Taking the difference, applying the triangle inequality, and taking expectations gives 
	\[
	\nu\, \mathbb{E} [\vert \hat\sigma^2 - \bar\sigma^2\vert \mathbf{1}_C]  \leq \mathbb{E}[\vert Y^{(1)T} (\bar M_1 - \hat M_1) Y^{(1)} \vert \mathbf{1}_C] + \mathbb{E}[\vert Y^{(2)T} (\bar M_2 - \hat M_2) Y^{(2)} \vert  \mathbf{1}_C].
	\]
	On the event $C$, $\bar M_j X^{(j)}\theta = 0$ and $\hat M_j X^{(j)}\theta = 0$ for $j\in \{1, 2\}$, thus
	\begin{equation}\label{eqEpsProj}
		\mathbb{E}[\vert Y^{(j)T} (\bar M_j - \hat M_j) Y^{(j)} \vert\mathbf{1}_C] =  \mathbb{E}[\vert \varepsilon^{(j)T} (\bar M_j - \hat M_j) \varepsilon^{(j)} \vert \mathbf{1}_C] \leq  \mathbb{E} \vert \varepsilon^{(j)T} (\bar M_j - \hat M_j) \varepsilon^{(j)} \vert. 
	\end{equation}
	%We will construct an upper bound for the right hand side. 
	Since $\varepsilon^{(j)}$ is independent of $\hat{\mathcal{E}}^{(-j)}$ (the latter is a function of $Y^{(-j)}$), we will condition on $\hat{\mathcal{E}}^{(-j)}$ in the construction of a bound for the right hand side of \eqref{eqEpsProj}. For a standard normal vector $Z$ and a non-random symmetric matrix $A$, $\mathbb{E}(Z^T A Z) = \text{tr}(A)$ and $\mathbb{E}[(Z^T A Z)^2] = \text{tr}(A)^2 + 2 \text{tr}(A^2)$, thus,
	\begin{eqnarray*}
		\mathbb{E}[ \varepsilon^{(j)T} (\bar M_j - \hat M_j) \varepsilon^{(j)}\mid \hat{\mathcal{E}}^{(-j)} ] &=& \sigma^2 \text{tr}(\bar M_j - \hat M_j) = 0,\\
		\mathbb{E}[ (\varepsilon^{(j)T} (\bar M_j - \hat M_j) \varepsilon^{(j)})^2\mid \hat{\mathcal{E}}^{(-j)} ] &=& 2  \sigma^4 \text{tr}\{(\bar M_j - \hat M_j)^2\} = 2 \sigma^4 [\text{tr}(\bar M_j^2) + \text{tr}(\hat M_j^2) -2\text{tr}(\bar M_j\hat M_j)  ] \leq 4\sigma^4 df_{-j}. 
	\end{eqnarray*}
	We have used that both projection matrices have the same trace by assumption. By Cauchy-Schwarz,
	\[
	\mathbb{E}[\vert \varepsilon^{(j)T} (\bar M_j - \hat M_j) \varepsilon^{(j)}\vert \mid \hat{\mathcal{E}}^{(-j)} ] \leq 2 \sigma^2 \sqrt{df_{-j}}.
	\]
	Since this holds for all $\hat{\mathcal{E}}^{(-j)}$, we obtain the bound
	\[
	\mathbb{E}\vert \varepsilon^{(j)T} (\bar M_j - \hat M_j) \varepsilon^{(j)}\vert \leq 2 \sigma^2 \sqrt{df_{-j}}.
	\]
	
	Combining these conclusions gives, by Markov's inequality,
	\[
	\pr\left\{\left\vert \frac{\hat\sigma^2}{\sigma^2} - \frac{\bar\sigma^2}{\sigma^2} \right\vert > \delta\right\} \leq \pr\left\{\left\vert \frac{\hat\sigma^2}{\sigma^2} - \frac{\bar\sigma^2}{\sigma^2} \right\vert > \delta, C\right\} + \varepsilon_n\leq \frac{2}{\delta \nu} (\sqrt{df_1} + \sqrt{df_2})  + \varepsilon_n,
	\]
	where $\varepsilon_n := 
	\pr(\mathcal{E}^*\not\subseteq \hat{\mathcal{E}}^{(1)} \cap \hat{\mathcal{E}}^{(2)})$. If $\vert\hat{\mathcal{E}}^{(1)} \vert + \vert\hat{\mathcal{E}}^{(2)} \vert= o(n)$, then $2(\sqrt{df_1} + \sqrt{df_2})/\nu = O(n^{-1/2})$. Since $\varepsilon_n = O(n^{-1/2})$ by the statement of Proposition \ref{propMRCV}, there exists a constant $A > 0$ such that
	\[
	\pr\left\{\left\vert \frac{\hat\sigma^2}{\sigma^2} - \frac{\bar\sigma^2}{\sigma^2} \right\vert > \delta\right\} \leq \left( \frac{1}{\delta} + 1 \right) \frac{A}{\sqrt{n}}.
	\]
	Upon setting $\delta = M n^{-1/2}$, we obtain
	\[
	\pr\left\{\left\vert \frac{\hat\sigma^2}{\sigma^2} - \frac{\bar\sigma^2}{\sigma^2} \right\vert > M n^{-1/2} \right\} \leq \left( \frac{\sqrt{n}}{M} + 1 \right) \frac{A}{\sqrt{n}} \to \frac{A}{M} \quad \text{as } n\to\infty.
	\]
	This establishes the stated order of distributional convergence.

\end{proof}

\subsection{Proof of Proposition \ref{thmQ}}

\begin{proof}
	Since
	\[
	\hat{\Gamma} - \Gamma = (\hat{\sigma}-\sigma)\begin{bmatrix}
	1 & 0 & 0 & \cdots & 0 \\
	\tilde a_1 & -\tilde b_1 & -\tilde b_1 &  \cdots & -\tilde b_1 \\
	0 & \tilde a_2 & -\tilde b_2 &  \cdots & -\tilde b_2 \\
	\vdots & \vdots & \vdots & \vdots & \vdots \\
	0 & 0 & 0 & \tilde a_{k-1} & -\tilde b_{k-1}
	\end{bmatrix},
	\]
	where $\tilde a_i \asymp k^{1/2}$ and $\tilde b_i \asymp k^{-1/2}$ are defined in equation \eqref{eqRecursive}, the components $\hat{Y}^{(1)}, \ldots, \hat{Y}^{(k)}$ can be written in terms of $\tilde{Y}^{(1)}, \ldots, \tilde{Y}^{(k)}$ as
	\begin{equation}\label{eqDecompY}
		\hat Y^{(i)}=\tilde{Y}^{(i)} + [Y \; L](\hat{\Gamma}_i - \Gamma_i) = \tilde{Y}^{(i)} + \delta  1_n 
	\end{equation}
	for all $i$, where $\delta=O(a_n \sqrt{k})$. 
	
	Since $U$ has singular values 0 and 1,
	\begin{align*}
		\|\hat Q^{(i)} - \tilde Q^{(i)}\|_2  \leq & \; \biggl\|\frac{\hat Y^{(i)}}{\|U^\T \hat Y^{(i)}\|_2} - \frac{\tilde Y^{(i)}}{\|U^\T \tilde Y^{(i)}\|_2}\biggr\|_2  \\
		= & \; \biggl\| \frac{\delta 1_{n}}{\|U^\T \tilde Y^{(i)}\|_2} - \frac{\hat Y^{(i)}}{\|U^\T \tilde Y^{(i)}\|_2}\biggl(\frac{\|U^\T \tilde Y^{(i)}\|_2}{\|U^\T \hat Y^{(i)}\|_2} -1\biggr)  \biggr\|_2 = O_p (a_n\sqrt{k}),
	\end{align*}
	where the last equality follows by Lemma \ref{lemmaRatio} and because $\|U^\T \tilde{Y}^{(i)}\|_2 \asymp (n-d_\theta)^{1/2}\sqrt{k}$ by orthonormality of the columns of $U$, which implies that the entries are of order $n^{-1/2}$ on average.
	
\end{proof}

\begin{lemma}\label{lemmaRatio}
	Under the conditions of Propisition \ref{thmQ}, 
	\[
	\frac{\|U^\T \tilde Y^{(i)}\|_2}{\|U^\T \hat Y^{(i)}\|_2} = 1 +  O_{p}(a_n \sqrt{k}).
	\]
\end{lemma}

\begin{proof}
	By \eqref{eqDecompY},
	\[
	U^\T\hat Y^{(i)}=U^\T\tilde{Y}^{(i)} + U^\T[Y \; L](\hat{\Gamma}_i - \Gamma_i) = U^\T\tilde{Y}^{(i)} + \delta_u  1_{n-d_\theta},
	\]
	where $\delta_u=O(a_n k^{1/2})$ by orthonormality of the columns of $U$, which implies that the entries are of order $n^{-1/2}$ on average. Thus, 
	\begin{equation}\label{eqUY}
		\frac{\|U^\T \tilde Y^{(i)}\|_2}{\|U^\T \hat Y^{(i)}\|_2} \leq \frac{\|U^\T \tilde Y^{(i)}\|_2}{\max\{\|U^\T \tilde Y^{(i)}\|_2 - \|\delta_u 1_{n-d_\theta}\|_2,0\}},
	\end{equation}
	and it only remains to show that that $\|U^\T \tilde Y^{(i)}\|_2 - \|\delta_u 1_{n-d_\theta}\|_2>0$ on an event of asymptotic probability 1. This follows on noting that $(\sigma k)^{-1}\|U^\T \tilde{Y}^{(i)}\|_2^2$ has a $\chi^2$ distribution on $n-d_\theta$ degrees of freedom. Thus, by Lemma 1 of Laurent and Massart (2000),
	\[
	\pr\Bigl((\sigma^2 k)^{-1}\|U^\T \tilde{Y}^{(i)}\|_2^2 > (n-d_\theta) - 2(n-d_\theta)^{1/2}\sqrt{t}\Bigr) >  1-\exp(-t)
	\]
	and consequently
	\[
	\pr\biggl(\bigcap_{i=1}^k\Bigl\{\|U^\T \tilde{Y}^{(i)}\|_2^2 > \sigma k(n-d_\theta)^{1/2}\Bigl[(n-d_\theta)^{1/2} - 2\sqrt{t}\Bigr]\Bigr\}\biggr) >  \{1-\exp(-t)\}^k 
	\]
	Thus $\|U^\T \tilde Y^{(i)}\|_2 \gtrsim a_n (k(n-d_\theta))^{1/2}$ for all $i$ with probability bounded below by a term of order
	\[
	\biggl\{1-\exp\biggl(-\frac{(n-d_\theta)a_n^2}{2k}\biggr)\biggr\}^k \rightarrow 1.
	\]
	It follows by a Taylor expansion for small $\|\delta_u 1_{n-d_\theta}\|_2=O_p(a_nk^{1/2})$ in the right-hand side of \eqref{eqUY} that
	\[
	\frac{\|U^\T \tilde Y^{(i)}\|_2}{\|U^\T \hat Y^{(i)}\|_2} = 1 + O_{p}(a_n k^{1/2}).
	\]
\end{proof}

\subsection{Proof of Theorem \ref{corrRayleigh}}

\begin{proof}
	Consider
	\begin{align*}
		\frac{U^\T \hat{Y}^{(i)}}{k^{1/2}\hat\sigma} &= \frac{U^\T \hat{Y}^{(i)}}{k^{1/2}\bar{\sigma}} \bigg(1 + \frac{\hat\sigma^2 - \bar\sigma^2}{\bar\sigma^2} \bigg)^{-1/2} \\
		&=: T^{(i)}(1 + O_p(n^{-1/2})),
	\end{align*}
	where the second equality follows by Proposition \ref{propMRCV}, and $\bar\sigma^2$ is defined in Equation \ref{eq:oracle-estimator}.
	We now consider a stochastic expansion of the $T^{(i)}$ term.
	Let $\nu=df_1+df_2$ as in the statement of the Theorem, then by  Equation \ref{eq:chi-square} in the proof of Proposition \ref{propMRCV} the entries of $T^{(i)}$ are student-$t$ distributed with $\nu$ degrees of freedom. Let $T$ denote an arbitrary random variable with this same distribution. Its density function at $t$ is
	\[
	f_{T}(t)dt = \frac{\Gamma\{(\nu+1)/2\}}{\Gamma(\nu/2)\sqrt{\pi \nu}}(1+t^2/\nu)^{-(\nu+1)/2}dt.
	\]
	By Stirling's formula in the form $\Gamma(x+a)/\Gamma(x)\simeq x^a$ for $a$ fixed and $x\rightarrow \infty$,
	\[
	\frac{\Gamma\{(\nu+1)/2\}}{\Gamma(\nu/2)} \simeq \left(\frac{\nu}{2}\right)^{1/2} \quad (\nu\rightarrow \infty).
	\]
	so that the constant term in $f_{T}$ is approximately $(2\pi)^{1/2}$ for large $\nu$. On taking logarithms and expanding for small $t^2/\nu$,
	\begin{align*}
		-\frac{(\nu+1)}{2}\log\Bigl(1+\frac{t^2}{\nu}\Bigr) =& -\frac{(\nu+1)}{2}\biggl\{\sum_{k=1}^{\infty}(-1)^{k+1}\frac{(t^{2}/\nu)^{k}}{k}\biggr\} \\
		=& -\frac{t^2}{2} + \nu^{-1}\biggl(\frac{t^4}{4} -\frac{t^2}{2}\biggr) + \nu^{-2}\biggl(\frac{t^4}{4} - \frac{t^6}{6}\biggr) + O(\nu^{-3}),
	\end{align*}
	and the density function of $T$ satisfies
	\[
	f_{T}(t) = \phi(t)\biggl\{1+ \nu^{-1}\biggl(\frac{t^4}{4} -\frac{t^2}{2}\biggr) + \nu^{-2}\biggl(\frac{t^4}{4} - \frac{t^6}{6}\biggr)\biggr\} + O(\nu^{-3})
	\]
	where $\phi$ is the standard normal density function. It follows that, for $Z$ a standard normal random variable, we can write the stochastic expansion 
	\begin{equation}\label{eqStochExpansion}
		T\stackrel{d}{=} Z + b_1(Z)/\nu^{1/2} + b_2(Z)/\nu + O_p(\nu^{-3/2}),
	\end{equation}
	where $b_1(z)$ can be obtained from equation (3.73) of Barndorff-Nielsen and Cox (1989) and $b_2(z)$ can be obtained from equation (3.72). 
	
	Since the above stochastic expansion holds for every entry of $T^{(i)}$ defined above, the entries being asymptotically independent, we can write 
	\begin{equation}\label{eqTZExpansion}
		T^{(i)}\stackrel{d}{=}Z^{(i)} +  B_1^{(i)}/\nu^{1/2} + B_2^{(i)}/\nu + O_p(\nu^{-3/2}),
	\end{equation}
	where $Z^{(i)}$ is an $(n-d_\theta)$-dimensional vector of independent standard normal random variables and $B_1^{(i)}$ and $B_2^{(i)}$ are random variables with entries distributed marginally as $b_1(Z)$ and $b_2(Z)$ respectively. 
	
	We will prove that the entries of the $(n-d_\theta)$-dimensional vectors $T^{(i)}$ are independent. Suppose for a contradiction that they are not. Then either $Z^{(i)}$ must have dependent entries, or $B_1^{(i)}/\nu^{1/2}$ or subsequent terms must have dependent entries. Since $\hat{\sigma}^2\rightarrow \sigma^2$ as $\nu\rightarrow \infty$, in this limit, $T^{(i)}$ is mean-zero Gaussian with covariance matrix $I_{n-d_\theta}$, contradicting dependence in the entries of $Z^{(i)}$. Independence in the entries of $Z^{(i)}$ contradicts independence of the entries of $B_1^{(i)}/\nu^{1/2}$, $B_2^{(i)}/\nu$, etc.~as the stochastic expansion \eqref{eqStochExpansion} shows that these only depend on the corresponding entry of $Z^{(i)}$. That both $\nu$ and the dimension of $T^{(i)}$ are increasing with $n$ does not affect this argument.

	Thus write
	\begin{align*}
		\bigl\langle T^{(i)}, T^{(j)}\bigr\rangle \stackrel{d}{=} & \;\bigl\langle Z^{(i)}, Z^{(j)}\bigr\rangle +  \bigl\langle Z^{(i)}, B_1^{(j)}/\nu^{1/2}\bigr\rangle + \bigl\langle Z^{(j)}, B_1^{(i)}/\nu^{1/2}\bigr\rangle  + \cdots \\
		= & \;\bigl\langle Z^{(i)}, Z^{(j)}\bigr\rangle + O_{p}(\max\{((n-d_\theta)\nu)^{-1/2}, \nu^{-1}\}) = O_p((n-d_\theta)^{-1}).
	\end{align*}
	By the discussion of \S \ref{secElliptic}, the squared lengths $\|T^{(i)}\|^2$ have the distribution of $\chi^2_{n-d_\theta}/V \stackrel{d}{=}(n-d_\theta) F_{n-d_\theta, \nu}$ where $V\stackrel{d}{=}\chi^2_{\nu}\nu$ is independent of the numerator $\chi^2_{n-d_\theta}$ and $F_{n-d_\theta, \nu}$ is a random variable with the $F$ distribution on $n-d_\theta$ and $\nu$ degrees of freedom. It follows that the inverse lengths are of order $\|T^{(i)}\|^{-1}=O_p((n-d_\theta)^{-1/2})$ uniformly in $i$. This allows control over the remainder terms. That $\bigl\langle T^{(i)}/\|T^{(i)}\|_2, T^{(j)}/\|T^{(j)}\|_2\bigr\rangle$ has the same distribution as $\bigl\langle Z^{(i)}/\|Z^{(i)}\|_2, Z^{(j)}/\|Z^{(j)}\|_2\bigr\rangle$ up to this order follows from the discussion of \S \ref{secElliptic}. It follows that
	\begin{align*}
		\hat{R}_n &\stackrel{d}{=} \big\{1 + O_p(n^{-1/2}) \big\} \bigr\{R_n + O_{p}\bigl(\max\bigl\{k\nu^{-1/2}(n-d_\theta)^{-1}, k\nu^{-1}(n-d_\theta)^{-1/2}\bigr\}\bigr) \bigr\}, 
	\end{align*}
	from which the statement of Theorem \ref{thmQ} follows.
\end{proof}

\section{An alternative proof of Proposition \ref{propPValueYanbo}}\label{appAlternativePolya}

As noted in the proof of Proposition \ref{propPValueYanbo}, the result as stated is a direct application of Polya's theorem. An alternative derivation is via an application of the following proposition, which replaces the $o(1)$ term in the statement of Proposition \ref{propPValueYanbo}. These rates may well be very conservative. On choosing $\zeta = \gamma^2/8$ in proposition \ref{prop:precise_rate} and noting that 
\begin{align*}
	\frac{b^{c^2/2\gamma + c\sqrt{\zeta/\gamma} - 1}}{\sqrt{\log b}} \rightarrow 0 \text{, and } 
	\frac{2}{b^{(2\sqrt{\zeta} - c/\sqrt{\gamma})^2/4}} \rightarrow 0
\end{align*}
for any $c <\sqrt{\gamma/2}$ we see that the error term in Proposition \ref{prop:precise_rate} tends to zero, recovering Proposition \ref{propPValueYanbo}.

\begin{proposition}\label{prop:precise_rate}
	Let $\hat\sigma^2$ be an estimator, independent of $R^2$, such that with $b_n = \gamma (n-d_\theta)$ for some $0 <\gamma \leq 1$, and $\epsilon_n   =  c\sqrt{\log(b_n)/b_n}$ for some $c < 2\sqrt{\zeta/\gamma}$ and $\zeta > 0$,
	\[
	\pr\Bigl(\Bigl|\frac{\hat{\sigma}^2}{\sigma^2} - 1\Bigr|\leq \epsilon_n\Bigr)\geq 1-r_n(\epsilon_n).
	\]
	Then, 
	\[
	\sup_{v>0}|\pr(\hat V \leq v) - \pr( V \leq v)| \leq 3r_n(\epsilon_n) + O\left( \frac{b^{c^2/2 + c\sqrt{\zeta} - 1}}{\sqrt{\log b}}\right) + O\left( \frac{2}{b^{(2\sqrt{\zeta} - c)^2/4}} \right).
	\]       
\end{proposition}

\begin{proof}
	Here we outline the argument. Let $\epsilon_n   = c \sqrt{\log(\gamma b)/\gamma b}$ with $c <  2\sqrt{\zeta/\gamma}$ and consider
	\begin{align*}
		\sup_{v \in \mathbb{R} }\bigl|\pr(\hat V \leq v) - \pr(V \leq v)\bigr| 
		&\leq \sup_{|v - b| < 2\sqrt{b\zeta\log (b)} + 2\zeta\log (b) }\bigl|\pr(\hat V \leq v) - \pr(V \leq v)\bigr| \\
		&+ \sup_{|v - b| \geq 2\sqrt{b\zeta\log (b)} + 2\zeta\log (b) }\bigl|\pr(\hat V \leq v) - \pr(V \leq v)\bigr|.
	\end{align*}
	We use a Riemann approximation on the bulk of the distribution for values of $v$ close to $b$, and control the truncation error for very large values of $v$, as the Riemann approximation becomes harder to control for large values of $v$. Lemmas \ref{lemmaPart1} and \ref{lemmaPart2} give
	\begin{align*}
		&\sup_{v \in \mathbb{R} }\bigl|\pr(\hat V \leq v) - \pr(V \leq v)\bigr|  \leq 3r_n(\epsilon_n) + O\left( \frac{b^{c^2/2\gamma + c\sqrt{\zeta/\gamma} - 1}}{\sqrt{\log b}}\right) +\frac{2}{b^{(2\sqrt{\zeta} - c/\sqrt{\gamma})^2/4}},
	\end{align*}
	as, in the notation defined there, $\pr(\mathcal{E}_{+}^c(\epsilon_n)) \leq \pr(\mathcal{E}^c(\epsilon_n)) = r_n(\epsilon_n)$. 
\end{proof}

\begin{lemma}\label{lemmaPart1}
	For $0<\epsilon_n <1$, introduce the events 
	\[
	\mathcal{E}_+(\epsilon_n) : \{\hat\sigma^2/\sigma^2 \leq 1+\epsilon_n\}, \quad \quad
	\mathcal{E}(\epsilon_n) := \{|\hat\sigma^2/\sigma^2 - 1| \leq \epsilon_n\}.
	\]
	Let $b = n -d_\theta$, the other notation being that of Proposition \ref{propPValueYanbo}. Then for $\epsilon_n   =  c\sqrt{\log(\gamma b)/\gamma b }$, for some $c < 2\sqrt{\zeta/\gamma}$ and $0 <\gamma \leq 1$,
	\begin{align*}
		&\sup_{|v - b| <   2\sqrt{\zeta b\log(b)} + 2\zeta\log(b)  } \bigl|\pr(\hat V \leq v) - \pr(V \leq v)\bigr| \\
		&\leq \pr(\mathcal{E}^c(\epsilon_n)) + 2\pr(\mathcal{E}_{+}^c(\epsilon_n)) + O\biggl( \frac{b^{c^2/2\gamma + c\sqrt{\zeta/\gamma} - 1}}{\sqrt{\log b}}\biggr).
	\end{align*}
\end{lemma}
\begin{proof}
	On the event $\mathcal{E}_+(\epsilon_n)$, $\{\hat V \leq v\}$ implies $\{V\leq v(1+\epsilon_n)\}$, giving the probability bound
	\begin{equation}\label{eqUpperPr}
		\pr(\hat V \leq v) \leq \pr(V\leq v(1+\epsilon_n))\,\pr(\mathcal{E}_{+}(\epsilon_n)) + \pr(\mathcal{E}_{+}^c(\epsilon_n))
	\end{equation}
	For a lower bound, write
	\[
	\hat V = V(\sigma^2/\hat{\sigma}^2 - 1) + V \leq V(1+ |\sigma^2/\hat{\sigma}^2 - 1|).
	\]
	Since $\mathcal{E}(\epsilon_n)$ can be written as
	\[
	\mathcal{E}(\epsilon_n) = \Bigl\{\frac{\epsilon_n}{1+\epsilon_n} \leq \frac{\sigma^2}{\hat\sigma^2} - 1 \leq \frac{\epsilon_n}{1-\epsilon_n} \Bigr\},
	\]
	on the event $\mathcal{E}(\epsilon_n)$,
	\[
	\{\hat V \leq v\} \supseteq \Bigl\{V \leq \frac{v}{1+\epsilon_n/(1-\epsilon_n)} \Bigr\} = \{V\leq v(1-\epsilon_n)\}.
	\]
	This implies the lower bound
	\begin{equation}\label{eqLowerPr}
		\pr(\hat V \leq v) \geq \pr(V\leq v(1-\epsilon_n))\,\pr(\mathcal{E}(\epsilon_n))
	\end{equation}
	
	By combining \eqref{eqUpperPr} and \eqref{eqLowerPr} 
	\begin{align*}   
		&\sup_{|v - b| <   2\sqrt{\zeta b\log(b)} + 2\zeta\log(b)}\bigl|\pr(\hat V \leq v) - \pr(V \leq v)\bigr| \\
		&\leq \sup_{|v - b| <   2\sqrt{\zeta b\log(b)} + 2\zeta\log(b)}\sup_{\bar v:|v-\bar v|\leq v\epsilon_n} U(v,\bar v,\epsilon_n),
	\end{align*}
	where 
	\begin{eqnarray*}
		U(v,\bar v,\epsilon_n) \hspace{-0.2cm} &=& \hspace{-0.2cm} \max\Bigl\{ \bigl|\pr(V\leq \bar v)\,\pr(\mathcal{E}(\epsilon_n)) - \pr(V \leq v) \bigr|, \\
		& &  \quad \quad \bigl|\pr(V\leq \bar v)\,\pr(\mathcal{E}_{+}(\epsilon_n)) - \pr(V \leq v) \Bigr| + \pr(\mathcal{E}_{+}^c(\epsilon_n))\bigr\}.
	\end{eqnarray*}
	The two terms in the maximum can be upper bounded as
	\begin{align*}
		\bigl|\pr(V\leq \bar v)\,\pr(\mathcal{E}(t)) - \pr(V \leq v) \bigr| \leq & \; \bigl|\pr(V\leq \bar v) - \pr(V \leq v) \bigr| + \pr(\mathcal{E}^c(\epsilon_n)), \\
		\bigl|\pr(V\leq \bar v)\,\pr(\mathcal{E}_{+}(t)) - \pr(V \leq v) \Bigr| + \pr(\mathcal{E}_{+}^c(t)) \leq & \; \bigl|\pr(V\leq \bar v) - \pr(V \leq v) \Bigr| + 2\pr(\mathcal{E}_{+}^c(\epsilon_n)).
	\end{align*}
	We thus need an upper bound for $|\pr(V\leq \bar v) - \pr(V \leq v)|$.
	Note that
	\[
	\pr( V \leq \bar{v}) = \pr\left( V \leq Tv \right), 
	\]
	for some value of $T$ in $(1 - \epsilon_n, 1 + \epsilon_n)$, and $V$ is a chi-square random variable with $b$ degrees of freedom. We consider four cases: 
	
	\noindent    \textbf{Case 1:} $T \geq 1$, $v \geq b$. Since the mode of the density of $V$ is at $b - 2$, we use the following Riemann approximation to the relevant intractable integral
	\begin{align*}
		&\pr\left( V \leq Tv \right) - \pr(V \leq v) 
		\\
		&= \int_{v}^{Tv} \frac{\lambda^{b/2}}{\Gamma(b/2)} x^{b/2-1} \exp\left( -\frac{x}{2} \right) dx    \\
		&\leq \int_{v}^{(1 +\epsilon_n)v} \frac{1}{2^b\Gamma(b/2)} x^{b/2-1} \exp\left( -\frac{x}{2} \right) dx\\
		&\leq \frac{1}{2^{b/2}\Gamma(b/2)} \sum^{ \lceil \epsilon_n v \rceil}_{  i = 0}  (v+i )^{b/2-1} \exp\left( -\frac{ v }{2} - \frac{i}{2} \right)\\
		&= \frac{v^{b/2-1}}{2^b\Gamma(b/2)} \exp\left( -\frac{ v }{2} \right) \sum^{ \lceil \epsilon_n v \rceil}_{  i = 0}  \left(1+ \frac{i}{v} \right)^{b/2-1} \exp\left(- \frac{i}{2} \right) \\
		&\leq \frac{v^{b/2-1}}{2^{b-1}\Gamma(b/2)} \exp\left( -\frac{ v }{2} \right) \sum^{ \lceil \epsilon_n v \rceil}_{  i = 0}  \left(1+ \frac{i}{v} \right)^{b/2} \exp\left(- \frac{i}{2} \right), 
	\end{align*}
	We bound the sum by
	\begin{align*}
		&\sum^{\lceil \epsilon_n v \rceil}_{  i = 0}  \biggl(1+ \frac{i}{v} \biggr)^{b/2} \exp\biggl(- \frac{i}{2} \biggr)\\
		&\leq 2 \sum^{\lceil \epsilon_n v \rceil}_{  i = 0} \biggl\{ \exp\biggl( \frac{1}{2} - \frac{1}{2} \biggr) \biggr\}^i  \leq 2\lceil \epsilon_n v \rceil
	\end{align*}
	as $(1 + i/v)^b \leq \exp(i)$ and $v \geq b$. It remains to upper bound the multiplicative term 
	\begin{align*}
		&2\lceil \epsilon_n v \rceil\frac{v^{b/2-1}}{2^{b/2-1}\Gamma(b/2)} \exp\left( -\frac{v}{2}\right)\\
		&\leq \frac{4\lceil \epsilon_n v \rceil}{\sqrt{\pi b}} \exp\biggl\{-\frac{v}{2} + \Bigl(\frac{b}{2} - 1\Bigr) \log v -\frac{b}{2}\log b + \frac{b}{2}\biggr\} \\
		&\leq \frac{4\lceil \epsilon_n v \rceil}{\sqrt{\pi b}} \exp\Bigl\{-\frac{v}{2} + \frac{b}{2} (\log(v/b) + 1)  - \log v\Bigr\}\\
		&\leq \frac{4\lceil \epsilon_n v \rceil}{\sqrt{\pi b}} \exp\left( - \log v \right)\\
		&\leq \frac{\lceil\epsilon v \rceil}{v \sqrt{\pi b}} \leq \frac{\epsilon_n }{\sqrt{\pi b}} + \frac{1 }{\sqrt{\pi} b^{3/2}} = O\biggl( \frac{\sqrt{\log b}}{b^{3/2}} \biggr).
	\end{align*}   
	
	\noindent \textbf{Case 2:} $T < 1$, $v \geq b$. By the log-convexity of the gamma distribution for all $x > 0$ and $t \in [0,1]$ the gamma density function $\gamma$ satisfies 
	\[ \gamma(x + t) \leq  \frac{1}{2^{b/2} \Gamma(b/2)} (x +1)^{b/2 - 1} \exp(-x/2).
	\]
	Based on this estimate the Riemann approximation to the relevant integral is
	\begin{align}\label{eqCase2}
		\nonumber     &\pr\left( V \leq v \right) - \pr(V \leq Tv) 
		\\
		\nonumber        &\leq \int^{v}_{(1 -\epsilon_n)v} \frac{1}{2^b\Gamma(b/2)} x^{b/2-1} \exp\left( -\frac{x}{2} \right) dx\\
		\nonumber        &\leq \frac{1}{2^{b/2}\Gamma(b/2)} \sum^{ \lceil \epsilon_n v \rceil}_{  i = 0}  ((1 -\epsilon_n)v+i+1)^{b/2-1} \exp\left( -\frac{ (1 -\epsilon_n)v }{2} - \frac{i}{2} \right)\\
		\nonumber        &= \frac{\{(1 -\epsilon_n)v\}^{b/2-1}}{2^{b/2}\Gamma(b/2)} \exp\left( -\frac{ (1 -\epsilon_n)v }{2} \right) \sum^{ \lceil \epsilon_n v \rceil}_{  i = 0}  \left(1+ \frac{i + 1}{(1 -\epsilon_n)v} \right)^{b/2-1} \exp\left(- \frac{i}{2} \right) \\
		&\leq \frac{\{(1 -\epsilon_n)v\}^{b/2-1}}{2^{b/2 -1}\Gamma(b/2)} \exp\left( -\frac{ (1 -\epsilon_n)v }{2} \right) \sum^{ \lceil \epsilon_n v \rceil}_{  i = 0}  \left(1+ \frac{i + 1}{(1 -\epsilon_n)v} \right)^{b/2} \exp\left(- \frac{i}{2} \right).
	\end{align}
	In this expression,
	\begin{align*}
		\left(1+ \frac{i + 1}{(1 -\epsilon_n)v} \right)^{b/2} &= \left(1+ \frac{i + 1}{b} \frac{b}{(1 -\epsilon_n)v} \right)^{b/2} \\
		&\leq \exp\left(\frac{ib}{2(1 -\epsilon_n)v } + \frac{b}{2(1 - \epsilon_n)v}\right) \\
		&\leq e^{1/4}  \exp\left(\frac{ib(1 +\epsilon_n)}{2(1 +\epsilon_n^2)v } \right\}\\
		&\leq e^{1/4}  \exp\left(\frac{ib(1 +\epsilon_n)}{2v } \right\},
	\end{align*}
	where the last inequality is because $v \geq b$. Thus 
	\[
	\left(1+ \frac{i + 1}{(1 -\epsilon_n)v} \right)^{b/2} \exp\left(- \frac{i}{2} \right) \leq e^{1/4}\exp(\epsilon_n i/2 )
	\]
	in equation \eqref{eqCase2} and the partial geometric series is bounded,  for large enough $b$, by
	\[
	\sum^{ \lceil \epsilon_n v \rceil}_{  i = 0}  \biggl(1+ \frac{i}{(1 -\epsilon_n)v} \biggr)^{b/2} \exp\biggl(- \frac{i}{2} \biggr) \leq \frac{  \exp( \frac{\lceil \epsilon_n v \rceil}{2} \epsilon_n ) - 1}{ \exp(\frac{1}{2} \epsilon_n )  -1} \leq 2e^{1/2}\epsilon_n^{-1}\exp\left( \frac{c^2\log(\gamma b)}{2\gamma} \right).
	\]
	This follows because $v \leq b + 2\sqrt{\zeta b\log b} + 2\zeta\log(b)$ and $v/b - 1 \rightarrow 0$. 
	
	Consider the multiplicative term in \eqref{eqCase2}. For $b$ sufficiently large,
	\begin{align*}
		&\frac{\{(1 -\epsilon_n)v\}^{b/2-1}}{2^{b/2 -1}\Gamma(b/2)} \exp\biggl\{ -\frac{ (1 -\epsilon_n)v }{2} \biggr\} \\
		&\leq  \frac{2}{\sqrt{\pi b}} \exp\biggl\{-\frac{(1 - \epsilon_n)v}{2} + \Bigl(\frac{b}{2} - 1\Bigr) \log( (1- \epsilon_n)v ) - \frac{b}{2}\log(b) + \frac{b}{2} - \frac{1}{12b + 1}  \biggr\}\\
		&\leq \frac{2}{\sqrt{\pi b}} \exp\biggl\{\left(\frac{b}{2}\right) \log\left\{ \frac{(1- \epsilon_n)v )}{b} \right\} +  \frac{b}{2} -\frac{(1 - \epsilon_n)v}{2} - \log((1-\epsilon_n)v) \biggr\}\\
		&\leq \frac{2}{\sqrt{\pi b}} \exp\biggl\{ \frac{b}{2}\Bigl( \log(1 - \epsilon_n) + \log(v/b) + 1 - \frac{(1-\epsilon_n)v}{b}\Bigr) - \log( (1- \epsilon_n)v )   \biggr\}\\
		&\leq\frac{2}{\sqrt{\pi b}} \exp\biggl\{ -\frac{\epsilon_n b}{2} (1 - \frac{v}{b}) - \log( (1- \epsilon_n)v )  \biggr\}\\
		&\leq \frac{2\exp(\epsilon_n\sqrt{\zeta b\log(b)} + \epsilon_n\zeta\log(b) )}{\sqrt{\pi b}(1-\epsilon_n)b } \\
		&\leq \frac{4\exp(c\sqrt{\zeta/\gamma}\log(b))}{\sqrt{\pi}b^{3/2}  },
	\end{align*}
	where we have used the lower bound for Stirling's approximation.
	Thus the overall error in Case 2 is
	\[
	O\biggl( \frac{b^{c^2/2\gamma + c\sqrt{\zeta/\gamma} - 1}}{\sqrt{\log b}} \biggr).
	\]
	
	\noindent \textbf{Case 3:} $T \geq 1$, $v < b$. Similar arguments give
	\begin{align*}
		&\pr\left( V \leq Tv \right) - \pr(V \leq v) 
		&\leq \frac{v^{b/2-1}}{2^{b/2 -1}\Gamma(b/2)} \exp\left( -\frac{ v }{2} \right) \sum^{ \lceil \epsilon_n v \rceil}_{  i = 1}  \left(1+ \frac{i + 1}{v} \right)^{b/2} \exp\left(- \frac{i}{2} \right),
	\end{align*}
	and,  for sufficiently large $b$,
	\[
	\left(1+ \frac{i + 1}{v} \right)^{b/2} \exp\left(- \frac{i}{2} \right) \leq e^{1/4}\exp\left\{  \{\sqrt{\zeta\log(b)}b^{-1/2}\} i/2 \right\},
	\]
	so that
	\begin{eqnarray*}
		\sum^{ \lceil \epsilon_n v \rceil}_{  i = 1}  \left(1+ \frac{i}{(1 -\epsilon_n)v} \right)^{b/2} \exp\left(- \frac{i}{2} \right) &\leq& 2 \exp\biggl(\frac{\lceil \epsilon_n v \rceil}{2}  b^{-1/2}\sqrt{\zeta\log b}  \biggr) \\
		&\leq & 4 \biggl(\frac{b}{\zeta\log b}\biggr)^{1/2} \exp \biggl( \frac{c\sqrt{\zeta} \log b}{2\sqrt{\gamma}} \biggr). 
	\end{eqnarray*}
	The multiplicative term satisfies
	\begin{eqnarray*}
		\frac{\{v\}^{b/2-1}}{2^{b/2 -1}\Gamma(b/2)} \exp\biggl( -\frac{v }{2} \biggr) 
		&\leq &\frac{2}{\sqrt{\pi b}} \exp\biggl\{ -\frac{\epsilon_n b}{2} \Bigl(1 - \frac{v}{b}\Bigr) - \log( (1- \epsilon_n)v )  \biggr\}\\
		&\leq & \frac{2}{\sqrt{\pi b}} \exp\left(  - \log( (1- \epsilon_n)v )  \right) \\
		&=&  \frac{2}{\sqrt{\pi b}(1-\epsilon_n)(b - \sqrt2{\zeta\log b} - 2\zeta\log b) }.
	\end{eqnarray*}
	Therefore the error for this case is
	\[
	O\biggl( \frac{b^{c\sqrt{\zeta}/2\sqrt{\gamma} - 1}}{\sqrt{\log b}} \biggr).
	\]
	
	\noindent \textbf{Case 4:} $T < 1$, $v < b$. In this case 
	\begin{align*}
		&\pr( V \leq v) - \pr(V \leq Tv) \\
		&\leq \frac{\{(1 -\epsilon_n)v\}^{b/2-1}}{2^{b/2 -1}\Gamma(b/2)} \exp\biggl( -\frac{ (1 -\epsilon_n)v }{2} \biggr) \sum^{ \lceil \epsilon_n v \rceil}_{  i = 0}  \biggl(1+ \frac{i}{(1 -\epsilon_n)v} \biggr)^{b/2} \exp\biggl(- \frac{i}{2} \biggr),
	\end{align*}
	and
	\[
	\biggl(1+ \frac{i}{(1 -\epsilon_n)v} \biggr)^{b/2} \exp\left(- \frac{i}{2} \right) \leq e^{1/2}\exp\biggl\{\frac{i}{2} (\epsilon_n + \{\sqrt{\zeta\log(b)}b^{-1/2}\} + \epsilon_n\{\sqrt{\zeta\log(b)}b^{-1/2}\})\biggr\}.
	\]
	Therefore,  for $b$ sufficiently large,
	\begin{align*}
		&\sum^{ \lceil \epsilon_n v \rceil}_{  i = 0}  \biggl(1+ \frac{i}{(1 -\epsilon_n)v} \biggr)^{b/2} \exp\left(- \frac{i}{2} \right) \\
		&\leq 2 \frac{  \exp\left\{ \frac{\lceil \epsilon_n v \rceil}{2} (\epsilon_n + \{\sqrt{\zeta\log(b)}b^{-1/2}\} + \epsilon_n\{\sqrt{\zeta\log(b)}b^{-1/2}\}  \right\} -1}{\exp\left\{ \frac{1}{2} (\epsilon_n + \{\sqrt{\zeta\log(b)}b^{-1/2}\} + \epsilon_n\{\sqrt{\zeta\log(b)}b^{-1/2}\} \right\} - 1} \\
		&\leq \frac{4e\sqrt{b}}{( c/\sqrt{\gamma} +\sqrt{\zeta})\log(\gamma b)} \exp\biggl( \frac{c^2/\gamma + c\sqrt{\zeta/\gamma} }{2} \log b \biggr).
	\end{align*}
	Since $v \geq b - 2\sqrt{\zeta b\log(b)} - 2\zeta\log(b)$, the multiplicative term satisfies
	\begin{align*}
		&\frac{\{(1 -\epsilon_n)v\}^{b/2-1}}{2^{b/2 -1}\Gamma(b/2)} \exp\biggl\{-\frac{ (1 -\epsilon_n)v }{2} \biggr\} \\
		&\leq  \frac{2}{\sqrt{\pi b}} \exp\biggl\{-\frac{(1 - \epsilon_n)v}{2} + \Bigl(\frac{b}{2} - 1\Bigr) \log( (1- \epsilon_n)v ) - \frac{b}{2}\log(b) + \frac{b}{2} - \frac{1}{12b + 1}  \biggr\}\\
		&\leq \frac{2}{\sqrt{\pi b}} \exp\biggl\{ \Bigl(\frac{b}{2}\Bigr) \log\Bigl(\frac{(1- \epsilon_n)v}{b} \Bigr) +  \frac{b}{2} -\frac{(1 - \epsilon_n)v}{2} - \log((1-\epsilon_n)v) \biggr\}\\
		&\leq \frac{2}{\sqrt{\pi b}} \exp\biggl\{ \frac{b}{2}\Bigl( \log(1 - \epsilon_n) + \log(v/b) + 1 - \frac{(1-\epsilon_n)v}{b}\Bigr) - \log( (1- \epsilon_n)v )   \biggr\}\\
		&\leq\frac{2}{\sqrt{\pi b}} \exp\biggl\{ -\frac{\epsilon_n b}{2} \Bigl(1 - \frac{v}{b} \Bigr) - \log( (1- \epsilon_n)v )  \biggr\}\\
		&\leq \frac{2}{\sqrt{\pi b}(1-\epsilon_n)\{b - 2\sqrt{\zeta b\log b} - 2\zeta\log(b)\} }  = O\Bigl( \frac{1}{b^{3/2}} \Bigr),
	\end{align*}
	giving an error of
	\[
	O\biggl( \frac{b^{\frac{c^2/\gamma + c\sqrt{\zeta/\gamma} }{2} - 1}}{\sqrt{\log b}}  \biggr),
	\]
	which is dominated by the error in Case 2.
\end{proof}

The following bounds are used to control the truncation error in Lemma \ref{lemmaPart2}.

\begin{proposition}[Laurant and Massart 2000, Lemma 1]\label{prop:massart}
	For a random variable $Z$ distributed as a Chi-square with $b$ degrees of freedom the following tail bounds hold:
	\begin{align*}
		\pr( Z \geq b +2 \sqrt{bt} + 2t ) \leq \exp(-t)\\
		\pr( Z \leq b -2 \sqrt{bt} ) \leq \exp(-t).
	\end{align*}
\end{proposition}

\begin{lemma}\label{lemmaPart2}
	Let $\epsilon_n   = c \sqrt{\log(\gamma b)/\gamma b}$ and $c <  2\sqrt{\zeta/\gamma}$, then
	\[
	\sup_{|v - b| \geq 2\sqrt{b\zeta\log(b)} + 2\zeta\log(b) }\bigl|\pr(\hat V \leq v) - \pr(V \leq v)\bigr| \leq \frac{2}{b^{(2\sqrt{\zeta} - c/\sqrt{\gamma})^2/4}}.
	\]
\end{lemma}
\begin{proof}
	%    For the upper tail, from the proof of Proposition 3.2 it is sufficient to bound:
	\begin{align*}
		&\sup_{v \geq b + 2\sqrt{\zeta b \log b} + 2\zeta\log b} \ \bigl|\pr(V\leq \bar v) - \pr(V \leq v) \bigr|  \\
		&\leq \sup_{v \geq b + 2\sqrt{\zeta b \log b} + 2\zeta\log b} \ \pr(V > (1 - \epsilon_n) v)  + \sup_{v \geq b + 2\sqrt{\zeta b \log b } + 2\zeta\log b} \ \pr(V > v) \\
		&\leq 2\pr\Bigl(V > (1 - \epsilon_n)\bigl\{b + 2\sqrt{\zeta b\log b} + 2\zeta\log b \bigr\}\Bigr) \\
		&= 2\pr\left(V >  b + (2\sqrt{\zeta} - c/\sqrt{\gamma} )\sqrt{b\log b} + (2\sqrt{\zeta} - c/\sqrt{\gamma})^2\log b\right) \leq \frac{2}{b^{(2\sqrt{\zeta} - c/\sqrt{\gamma})^2/4}},
	\end{align*}
	since $\epsilon_n b   =  c(b\log (b)/\gamma)^{1/2}$  and $\epsilon_n\{2\sqrt{\zeta b\log b} + 2\zeta\log b\}< c\log(b)/\gamma$ eventually for sufficiently large $b$. By a similar argument for the left tail we have
	\begin{align*}
		&\sup_{v \leq b - 2\sqrt{\zeta b\log b } - 2\zeta\log b} \ \bigl|\pr(V\leq \bar v) - \pr(V \leq v) \bigr|  \\
		&\leq \sup_{v \leq b - 2\sqrt{\zeta b \log b } - 2\zeta\log b} \ \pr(V < (1 + \epsilon_n) v)  + \sup_{v \leq b - 2\sqrt{\zeta b\log b}- 2\zeta\log b} \ \pr(V < v) \\
		&\leq 2\pr\Bigl((V < (1 + \epsilon_n)\{b - 2\sqrt{\zeta b \log b} - 2\zeta\log b \}\Bigr) \\
		&= 2\pr\Bigl(V <  b - (2\sqrt{\zeta} - c/\sqrt{\gamma})\sqrt{b\log b}\Bigr) \leq \frac{2}{b^{(2\sqrt{\zeta} - c/\sqrt{\gamma})^2/4}}.
	\end{align*}
\end{proof}

\bibliographystyle{amsplain}

\end{appendix}

\bigskip\bigskip
\bigskip\bigskip
\bigskip\bigskip

\end{document}